\documentclass{amsart}
\usepackage[foot]{amsaddr}
\usepackage{graphicx}

\usepackage{amsfonts}
\usepackage{amsmath}
\usepackage{amssymb}
\usepackage{amsthm} % included in ntheorem with amsthm option
\usepackage{array}
\usepackage{float}
\usepackage[toc,page,titletoc]{appendix}
\usepackage{longtable}
\usepackage{graphicx}
\usepackage[font=singlespacing]{caption}
\usepackage[normalem]{ulem}
\usepackage{nicefrac}
\usepackage{units}
\usepackage{rotating}
\usepackage{afterpage}
\usepackage{tikz}
\usepackage{tikz-cd}
\usetikzlibrary{arrows,positioning}

\usepackage{mathrsfs}
\usepackage{amscd}

\usepackage{cleveref}
\usepackage{comment}
\usepackage{upgreek}
\usepackage{ragged2e}

\usepackage[mathscr]{euscript}
\usepackage{scalerel,stackengine}

\usepackage{enumitem}

\newtheorem{theorem}{Theorem}[section]
\newtheorem{lemma}[theorem]{Lemma}

\theoremstyle{definition}

\theoremstyle{remark}

\numberwithin{equation}{section}

\newtheorem{corollary}[theorem]{Corollary}

\theoremstyle{definition}

\newcommand{\overbar}[1]{\mkern 1.5mu\overline{\mkern-1.5mu#1\mkern-1.5mu}\mkern 1.5mu}
\def\til#1{\widetilde{#1}} %$\til{x}$ $\til{p}$

\newcommand{\mf}[1]{\mathfrak{#1}}
\newcommand{\ms}[1]{\mathscr{#1}}

\stackMath
\newcommand\reallywidehat[1]{%
\savestack{\tmpbox}{\stretchto{%
  \scaleto{%
    \scalerel*[\widthof{\ensuremath{#1}}]{\kern-.6pt\bigwedge\kern-.6pt}%
    {\rule[-\textheight/2]{1ex}{\textheight}}%WIDTH-LIMITED BIG WEDGE
  }{\textheight}% 
}{0.5ex}}%
\stackon[1pt]{#1}{\tmpbox}%
}

\newcounter{parno}[paragraph]%% numbered paragraph

\newcounter{subparno}[section]%% numbered paragraph

\newcommand\NoIndent[1]{%
  \begingroup
  \par
  \parshape0
  #1\par
  \endgroup
}

\begin{document}

\title{Morita Equivalence of $W^{*}$-Correspondences and Their Hardy Algebras}

%    Information for first author
\author{Rene Ardila}
%\author{Author One}
%    Address of record for the research reported here
\address{Department of Mathematics, Grand Valley State University, Allendale, Michigan 49401}
%    Current address
%\curraddr{Department of Mathematics,
%Grand Valley State University,  Allendale, Michigan 49401}
\email{ardilar@gvsu.edu}
%    \thanks will become a 1st page footnote.
%\thanks{The first author was supported in part by NSF Grant \#000000.}

%    Information for second author
%\author{Author Two}
%\address{Mathematical Research Section, School of Mathematical Sciences,
%Australian National University, Canberra ACT 2601, Australia}
%\email{two@maths.univ.edu.au}
%\thanks{Support information for the second author.}

%    General info

\date{May 20, 2018}

%\dedicatory{This paper is dedicated to our advisors.}


\begin{abstract}
Muhly and Solel developed a notion of Morita equivalence for $C^{*}$- correspondences, which they used to show that if two $C^{*}$-correspondences $E$ and $F$ are Morita equivalent then their tensor algebras $\mathcal{T}_{+}(E)$ and $\mathcal{T}_{+}(F)$ are (strongly) Morita equivalent operator algebras. We give the weak$^{*}$ version of this result by considering (weak) Morita equivalence of $W^{*}$-correspondences and employing Blecher and Kashyap's notion of Morita equivalence for dual operator algebras. More precisely, we show that weak Morita equivalence of $W^{*}$-correspondences $E$ and $F$ implies weak Morita equivalence of their Hardy algebras $H^{\infty}(E)$ and $H^{\infty}(F)$. We give special attention to $W^{*}$-graph correspondences and show a number of results related to their Morita equivalence. %We study  how different representations of a $W^{*} $-algebra give rise to Morita equivalent objects. For example, we show that if $(E,A)$ is a $W^{*}$-graph correspondence and we have two faithful normal representations $\sigma$ and $\tau$ of $A$, then the commutants of the induced representions $\sigma ^{\ms{F}(E)}(H^{\infty}(E))$ and $\tau ^{\ms{F}(E)}(H^{\infty}(E))$ are weakly Morita equivalent dual operator algebras. We also develop a categorical approach to Morita equivalence of $W^{*}$- correspondences. This involves building categories of covariant representations and studying the groups $Aut(\mathbb{D}({(E^{\sigma}})^*)$ and $Aut(H^{\infty}(E))$ (the automorphism groups of the unit ball of intertwiners and the Hardy algebra). In this regard, we advance the work of Muhly and Solel by showing new results about these groups, their matrix representation and their algebraic properties.
\end{abstract}

\maketitle

%\section*{This is an unnumbered first-level section head}
%This is an example of an unnumbered first-level heading.

%% The correct journal style for \specialsection is all uppercase; a known bug
%% in amsart.cls prevents this, so input must be uppercase until it is fixed.
%\specialsection*{This is a Special Section Head}
%\specialsection*{Preliminaries}
%This is an example of a special section head%
%%%%%%%%%%%%%%%%%%%%%%%%%%%%%%%%%%%%%%%%%%%%%%%%%%%%%%%%%%%%%%%%%%%%%%%%

%\footnote{Here is an example of a footnote. Notice that this footnote
%text is running on so that it can stand as an example of how a footnote
%with separate paragraphs should be written.
%\par
%And here is the beginning of the second paragraph.}%
%%%%%%%%%%%%%%%%%%%%%%%%%%%%%%%%%%%%%%%%%%%%%%%%%%%%%%%%%%%%%%%%%%%%%%%%
%.

\section{Introduction}
%This study derives from concerns the study of Morita equivalence of $W^{*} $-correspondences and their Hardy algebras, with special attention given to $W^{*} $-correspondences derived from directed graphs. 
%The idea of \emph{$W^{*} $-correspondence} was introduced by Connes in \cite{Connes1982} as a tool to study the classification of factors. %more specifically, he used corresondences to define the property T of Kashdan for arbitrary type II_{1} factors
%In his unpublished notes, and later in \cite{Connes1990}, he also presented $W^{*}$-correspondences as a notion of morphism between von Neumann algebras. 
Given a von Neumann algebra $A$ and a $W^{*}$-correspondence $E$ over $A$, Muhly and Solel constructed an algebra $H^{\infty}(E)$ which they called the Hardy algebra of $E$ \cite{Muhly2004a}. This algebra is a noncommutative generalization of the classic Hardy algebra $H^{\infty}(\mathbb{T})$ of bounded analytic functions on the open unit disc. More precisely, when $E=A=\mathbb{C}$, $H^{\infty}(E)$ is the classical Hardy space $H^{\infty}(\mathbb{T})$. When $A=\mathbb{C}$ and $E=\mathbb{C}^{n}$, $H^{\infty}(E)$ is the free semigroup algebra $\mathfrak{L}_{n}$ studied by Popescu \cite{Popescu1991}, Davidson, Pitts \cite{DPit98a} and others. %As we will see in the next section, 
This Hardy algebra is a dual operator subalgebra of $\ms{L}(\ms{F}(E))$, the adjointable operators of the Fock space of $E$, generated by diagonal and creation operators. When $E$ is a correspondence derived from a directed graph $G$, $H^{\infty}(E)$ is a dual operator algebra version of what algebraists call the path algebra of $G$.

Kiiti Morita's 1958 groundbraking %influential
paper \cite{Morita1958} contains the main ideas of what later became known as Morita equivalence, an extremely important concept in the study of the algebraic structure of rings. Following the dissemination of Morita's ideas, mainly by H. Bass and P. Gabriel in the early 1960s, many other notions of Morita equivalence have been developed, including notions of Morita equivalence for  
%Among others, there are notions  (for objects-in order to study the structure of algebraic objects- such as
selfadjoint algebras, operator algebras, groupoids, group *-algebras, finite groups, Poisson manifolds, non commutative smooth tori, tensor categories, semigroups and star products.
%in different fields of mathematics, such as... The general theme in all these notions is the existence of an equivalence between the categories of representations of the objects involved.
In \cite{Muhly2000a}, Muhly and Solel introduced a notion of (strong) Morita equivalence for $C^{*}$-correspondences, which they used to show that if two $C^{*}$-correspondences $E$ and $F$ are (strongly) Morita equivalent then their tensor algebras $\mathcal{T}_{+}(E)$ and $\mathcal{T}_{+}(F)$ are (strongly) Morita equivalent operator algebras. At that time however, there was no clear notion of Morita equivalence for dual operator algebras. Such notions were developed ten years later in the work of Blecher, Kashyap, Eleftherakis and Paulsen (\cite{Blecher2008a}, \cite{Eleftherakis2008}, \cite{Eleftherakis2008b}). Motivated by Muhly and Solel's work, we consider 
%extend the notion of (strong) Morita equivalence of $C^{*}$-correspondences to
(weak) Morita equivalence of $W^{*}$-correspondences, and use Blecher and Kashiap's notion of Morita equivalence for dual operator algebras to show that if two $W^{*}$-correspondences $E$ and $F$ are (weakly) Morita equivalent then their Hardy algebras $H^{\infty}(E)$ and $H^{\infty}(F)$ are (weakly) Morita equivalent dual operator algebras.

In %Chapter \ref{graphChap},
the last section, 
 we concentrate on Morita equivalence of $W^{*}$-graph correspondences, their Hardy algebras and their representations. We show that if $(E,A)$ is a $W^{*}$-graph correspondence then any two faithful normal representations $\sigma$ and $\tau$ of $A$ give rise to Morita equivalent dual corresoondences $(E^{\sigma}$, $\sigma(A)')$ and $(E^{\tau}$, $\tau(A)')$. Then we consider the induced representations $\sigma ^{\ms{F}(E)}$ and $\tau ^{\ms{F}(E)}$ of the Hardy algebra $H^{\infty}(E)$ and show that the commutants of $\sigma ^{\ms{F}(E)}(H^{\infty}(E))$ and $\tau ^{\ms{F}(E)}(H^{\infty}(E))$ are (weakly) Morita equivalent dual operator algebras. We also study  equivalence bimodules and the relation between graphs and the Morita equivalence of their $W^{*}$-correspondences. %In Chapter \ref{categorical}, building on Muhly and Solel's work \cite{Muhly2011b}, we present a categorical approach to Morita equivalence of $W^{*}$-correspondences.
 
\section{Preliminaries} \label{preliminaries}
%We assume that the reader is familiar with the basic theory of Banach spaces, $C^{*}$-algebras and the basic notions of operator space theory. Let $A$ be a $C^{*}$-algebra. A right \emph{inner product $A$-module}  is a vector space $E$ wich is a right $A$-module (with compatible scalar multiplication: $\lambda(x\cdot a)=(\lambda x)\cdot a=x\cdot(\lambda a)$ for $x\in E$, $a\in A$, $\lambda \in \mathbb{C}$.), together with an $A$-valued inner product $<\cdot,\cdot>:E \times E \to A$ satisfying:
%\begin{enumerate}
%\item $\langle x, \alpha y+\beta z \rangle =\alpha \langle x,y \rangle +\beta \langle x,z \rangle \qquad (x,y,z \in E, \alpha,\beta \in \mathbb{C}$)
%\item $\langle x,y \cdot a\rangle =\langle x,y\rangle a \qquad (x,y \in E, a \in A$)
%\item $\langle x,y\rangle ^{*}=\langle y,x\rangle \qquad (x,y \in E$)
%\item $\langle x,x\rangle \thinspace \geq 0$ \qquad (with $\langle x,x\rangle =0$ only when $x=0$)
%\end{enumerate} 
%A right Hilbert $C^{*}$-module over $A$ (or Hilbert $C^{*}$-$A$-module) is an inner product $A$-module, which is complete under the norm $\left\Vert x  \right\Vert := \left\Vert \langle x,x\rangle  \right\Vert ^{1/2}$.
%
A right $C^{*}$-module $E$ over a $C^{*}$-algebra $A$ is said to be \emph{selfdual} if every continuous $A$-module map $f:E\to A$ is of the form $f(\cdot)=\langle y, \cdot\rangle$, for some $y\in E$.  We say that $E$ is a right \emph{$W^{*}$-module} if $E$ is a selfdual right $C^{*}$-module over a $W^{*}$-algebra.
%A right \emph{$W^{*}$-module} is a selfdual right $C^{*}$-module $E$ over a $W^{*}$-algebra.
We write $\ms{L}_{A}(E)$ (or simply $\ms{L}(E)$) for the space of \emph{adjointable} $A$-module maps on $E$. %(\ref{adjointablemaps}).
 An $A-B$ \emph{$W^{*}$-correspondence} is a right $W^{*}$-module $E$ over $B$ for which there exists a unital normal $*$-homomorphism $\varphi$: $A \to \ms{L}_{B}(E)$. We then say that $E$ is a $W^{*}$-correspondence from $A$ to $B$, and we denote it by $_{A}E_{B}$. If $A=B$ then we say that $E$ is a $W^{*}$-correspondence over $A$. In this case, we might also denote the correspondence by $(E,A)$. The center of a $W^{*}$-correspondence $(E,A)$ is the set $\mathfrak{Z}(E)=\{x\in E : a\cdot x=x\cdot a\text{ for all } a\in A\}$. We will sometimes abbreviate "weak$^{*}$" to "$w^{*}$".

The \emph{$W^{*}$-module tensor product} $\overbar{\otimes}_{A}$ (sometimes written as the \emph{composition tensor product} $X\overbar{\otimes}_{\sigma}Y$) is defined to be the selfdual completion (the weak$^{*}$-completion) of the $C^{*}$-module interior tensor product $X\otimes_{\sigma}Y$. When there is no risk of confusion, we will simply write $X\overbar{\otimes}Y$. %This tensor product coincides with the module weak$^{*}$ Haagerup tensor product $\otimes _{w^{*}hA}$ (\ref{weak*tensorproduct}). 
%Unless stated otherwise, we shall assume all tensor products in this thesis refer to the $W^{*}$-module tensor product. 
The $W^{*}$-module tensor product is functorial and associative. If $E$ is a $W^{*}$-correspondence from $A$ to $B$ and $F$ is a $W^{*}$-correspondence from $B$ to $C$ then $E\overbar{\otimes}_{A}F$ is a $W^{*}$-correspondence from $A$ to $C$ with inner product given by $\left\langle x_{1} \otimes y_{1},x_{2} \otimes y_{2} \right\rangle _{E\otimes F} = \left\langle y_{1}, \varphi (\left\langle x_{1},x_{2}\right\rangle _{E})y_{2} \right\rangle_{F}$ and left/right actions given by $a\cdot(x\otimes y)\cdot c = (a\cdot x) \otimes (y \cdot c)= (\varphi_{E}(a)x) \otimes (y \cdot c)$. In particular, given a $W^{*}$-correspondence $E$ over $A$ and a Hilbert space $H$ equipped with a normal representation $\sigma$ of $A$, we can form the Hilbert space $E\overbar{\otimes}_{\sigma}H$, where we have $\left\langle x_{1} \otimes h_{1},x_{2} \otimes h_{2} \right\rangle = \left\langle h_{1}, \sigma (\left\langle x_{1},x_{2}\right\rangle _{E})h_{2} \right\rangle$.

A \emph{$W^{*}$-correspondence isomorphism} between two $W^{*}$-correspondences $(E_{1},A_{1})$ and $(E_{2},A_{2})$ is a pair $(\sigma, \psi)$ where $\sigma:A_{1}\to A_{2}$ is an isomorphism of $W^{*}$-algebras and $\psi:E_{1}\to E_{2}$ is a vector space isomorphism, where for $e,f\in E_{1}$ and $a,b\in A_{1}$, we have $\psi(a\cdot e\cdot b)=\sigma(a)\cdot\psi (e)\cdot \sigma(b)$ and $\langle\psi(e), \psi(f)\rangle =\sigma (\langle e,f\rangle )$. Such $\psi$ must be a weak$^{*}$-homeomorphism because the predual of a $W^{*}$-module is unique. %Thus, whenever we prove that a certain pair $(\sigma, \psi)$ is a $W^{*}$-correspondence isomorphism, we will never check that $\phi$ preserves the $\sigma$-topology.%Jenni's Appendix B.9 and B.11 can also be used.Thus, whenever we prove that a certain pair $(\sigma, \psi)$ is a $W^{*}$-correspondence isomorphism, we will never check that $\psi$ preserves the $\sigma$ topology.

%If $E_{1}$ is an $A_{1}-B_{1}$ $W^{*}$-correspondence and $E_{2}$ is an $A_{2}-B_{2}$ $W^{*}$-correspondence, then a \emph{$W^{*}$-correspondence isomorphism} from $_{A_{1}}{E_{1}}_{B_{1}}$ to $_{A_{2}}{E_{2}}_{B_{2}}$ is a triple $(\sigma, \phi, \pi)$, where $\sigma: A_{1}\to A_{2}$ and $\pi:B_{1}\to B_{2}$ are $W^{*}$-algebra isomorphisms and $\phi:E_{1}\to E_{2}$ is a vector space isomorphism such that $\phi(a\cdot e\cdot b)=\sigma(a)\cdot \phi(e)\cdot \pi(b)$ and $\langle\phi(e),\phi(f)\rangle_{E_{2}}=\sigma(\langle e,f \rangle_{E_{1}})$. Such $\phi$ must be a weak$^{*}$-homeomorphism because the predual of a $W^{*}$-module is unique. %Jenni's Appendix B.9 and B.11 can also be used.
%Thus, whenever we prove that a certain triple $(\sigma, \phi, \pi)$ is a $W^{*}$-correspondence isomorphism, we will never check that $\phi$ preserves the $\sigma$ topology. Note that if $B_{1}\cong A_{2}$ then we have $A_{1}\cong A_{2}\cong B_{1}\cong B_{2}$. So we must have that $\sigma =\pi$, so that the tensor product $E_{1}\overbar{\otimes}_{B_{1}}E_{2}$ is indeed balanced over $B_{1}$.

 If $A$ and $B$ are $W^{*}$-algebras, then an $A$-$B$  \emph{$W^{*}$-equivalence bimodule} is an $A$-$B$ $W^{*}$-bimodule $X$ which is a $w^{*}$-full right $W^{*}$-module over $B$ and a $w^{*}$-full left $W^{*}$-module over $A$, such that the two (left and right) inner products of $X$ are compatible in the sense that $_{A}\langle x,y \rangle \cdot z=x\cdot \langle y,z \rangle _{B} $ for all $x,y,z \in X$. If $_{A}X_{B}$ and $_{C}Y_{D}$ are $W^{*}$-equivalence bimodules then a \emph{$W^{*}$-equivalence bimodule isomorphism} (as defined in \cite[Definition 1.16 and Remark 1.19]{Echterhoff2006}) is a triple $(\sigma, \phi, \pi)$, where $\sigma:A\to C$ and $\pi:B \to D$ are $W^{*}$-algebra isomorphisms and $\phi:X\to Y$ is a vector space isomorphism such that $\phi(a\cdot e\cdot b)=\sigma(a)\cdot \phi(e)\cdot \pi(b)$, $\langle\phi(e),\phi(f)\rangle_{D}=\pi(\langle e,f \rangle_{B})$ and $_{C}\langle\phi(e),\phi(f)\rangle= \sigma(_{A}\langle e,f \rangle)$. %Again, we don't need to check that $\phi$ preserves the $\sigma$ topology, as this is guaranteed by the fact that $E_{1}$ and $E_{2}$ have unique preduals.

Given a representation $\sigma$ of $A$, an operator $T \in \ms{L}(E)$ and an operator $S\in \sigma(A)'$, the map $x \otimes h \to Tx \otimes Sh$ defines a bounded operator on $E\overbar{\otimes}_{\sigma}H$ denoted by $T\otimes S$. In particular, the representation of $\ms{L}(E)$ resulting from letting $S=I$, is Rieffel's induced representation of $\ms{L}(E)$ induced by $\sigma$. This representation is denoted by $\sigma ^{E}$. That is, $\sigma ^{E}(T)=T\otimes I$. Likewise, we say that the composition $\sigma ^{E} \circ \varphi$ is the representation of $A$ on  $E\overbar{\otimes}_{\sigma}H$ induced by $E$.

Let $\mathbb{N}_{0}=\mathbb{N} \cup \{0\}$. If $E$ is a $W^{*}$-correspondence over a $W^{*}$-algebra $A$ then we can form the tensor powers $E^{\overbar{\otimes}n}$, $n\geq 0$, where $E^{\overbar{\otimes}0}=A$. For each $n$, $E^{\overbar{\otimes}n}$ is a $W^{*}$-correspondence over $A$ with the inner product defined inductively. 
The ultraweak direct sum $\ms{F}(E):=\bigoplus^{wc}_{n\in \mathbb{N}_{0}}E^{\overbar{\otimes}n}$ is a $W^{*}$-correspondence over $A$ called the \emph{Fock space} over $E$. % (see \ref{ultraweak} for a definition of ultraweak direct sum).
 The left action of $A$ on $\ms{F}(E)$ is given by the map $\varphi_{\infty}$ defined by $\varphi_{\infty}(a)=$ diag$(a,\varphi(a),\varphi ^{(2)}(a) ,\varphi ^{(3)}(a)   , \cdots)$ where $\varphi^{(n)}(a)(x_{1} \otimes x_{2}\otimes \cdots\otimes x_{n})=(\varphi(a)x_{1}) \otimes x_{2}\otimes \cdots\otimes x_{n} \in E^{\overbar{\otimes}n}$. Given $x\in E$, the \emph{creation operator} $T_{x}\in \ms{L}(\ms{F}(E))$ is defined by $T_{x}(\eta)=x\otimes \eta$, $\eta \in \ms{F}(E))$. That is, 
\begin{center}
$
\varphi_{\infty}(a) = 
\left( 
\begin{matrix} 
\varphi^{(0)}(a) & 0  & \cdots  \\
0 & \varphi^{(1)}(a) &  \cdots  \\
0 & 0 & \ddots 
\end{matrix}
\right)$ and 
$
T_x = 
\left( 
\begin{matrix} 
0 & 0  & \cdots & 0 \\
T_x^{(1)} & 0 & \cdots & 0 \\
0 & T_x^{(2)} & \cdots & 0 \\
0 & 0 & \ddots & 0
\end{matrix}
\right)$
\end{center}
where $T_{x}^{(n)}(x_{1} \otimes x_{2}\otimes \cdots\otimes x_{n-1})=x \otimes x_{1} \otimes x_{2}\otimes \cdots\otimes x_{n-1}$.
%and $T_{x}(b)=x \cdot b$ for $b\in A$. So that $T_{x}$ maps $E^{\otimes n-1}$ into $E^{\otimes n}$. For each $x\in E$, $||T_{x}||=||x||$. The adjoint of $T_{x}$ is given by $(T_{x}^{(n)})^{*}(x\otimes \eta)=\varphi^{(n)}(\langle x,x\rangle )\eta$, $x\in E$, $\eta \in E^{\otimes n}$. That is,
%\begin{equation*}
%T_x^{*} = 
%\left( 
%\begin{matrix} 
%0  & (T_x^{(1)})^{*} & 0 & \cdots & 0 \\
%0 & 0 & (T_x^{(2)})^{*}  &\cdots & 0 \\
%0 & 0 & 0 & \ddots & 0 \\
%0 & 0 & 0 & \cdots & 
%\end{matrix}
%\right)
%\end{equation*}

The \emph{tensor algebra} over $E$, denoted $\mathcal{T}_{+}(E)$ is defined to be the norm closed subalgebra of $\ms{L}(\ms{F}(E))$ generated by $\varphi_{\infty}(A)$ and $\{T_{x}:x \in E \}$. The ultraweak closure of $\mathcal{T}_{+}(E)$ in $\ms{L}(\ms{F}(E))$ is called the \emph{Hardy Algebra} of $E$, and is denoted by $H^{\infty}(E)$.

As shown in \cite{Muhly2004a}, the completely contractive representations of $H^{\infty}(E)$ are determined by pairs $(T,\sigma)$ where $\sigma:A\to B(H)$ is a normal $*$-representation of $A$ and $T:E\to B(H)$ is a linear, completely contractive $w^{*}$-continuous representation of $E$ satisfying $T(axb)=\sigma(a)T(x)\sigma(b)$ for all $x\in E$ and $a,b\in A$. The linear map $\widetilde{T}$ defined on the algebraic tensor product $E\otimes H$ by $\widetilde{T}(x\otimes h)=T(x)h$ extends to an operator of norm at most 1 on the completion $E\overbar{\otimes}_{\sigma}H$. The pairs $(T,\sigma)$ are called the \emph{completely contractive covariant representations of $E$}. The bimodule property of $T$ is equivalent to the equation $\widetilde{T}(\sigma^{E}\circ \varphi(a))= \widetilde{T}(\varphi(a)\otimes I)=\sigma(a)\widetilde{T}$ for all $a\in A$, which means that $\widetilde{T}$ intertwines the representations $\sigma$ and $\sigma^{E}\circ \varphi$ of $A$ on $H$ and $E\otimes H$ respectively. The space composed of all these intertwiners is called the intertwining space, and it is usually denoted as $\mathcal{I}(\sigma ^{E} \circ \varphi, \sigma)$ or $(E^{\sigma})^{*}$. Furthermore, for each completely contractive covariant representation $(T,\sigma)$ of a correspondence $E$ over a $W^{*}$ algebra $A$, there is a unique completely contractive representation $\rho$ of the algebra $\mathscr{T} _{+}(E)$ satisfying $\rho (T_{x})=T(x)$ and $\rho(\varphi _{\infty}(a))=\sigma(a)$ for each $x\in E, a \in A$. The map $(T,\sigma)\mapsto \rho$ is bijective and onto the set of all completely contractive representations of $\mathscr{T} _{+}(E)$ whose restrictions to $\varphi(A)$ are ultraweakly continuous. If $||\widetilde{T}|| < 1$ then $\rho$ extends to an ultraweakly continuous representation {$\sigma \times \widetilde{T}$} of $H^{\infty}(E)$.  
That is, the ultraweakly continuous completely contractive representations of $H^{\infty}(E)$ are parametrized by the elements in the unit ball
\begin{equation*}
\mathbb{D}((E^{\sigma})^{*})=\{\eta^{*}\in B(E\otimes_{\sigma}H, H)\enspace | \enspace \eta^{*}(\varphi(a)\otimes I)=\sigma(a)\eta^{*} \text{and } ||\eta||<1\}
\end{equation*}
%
% Using the representation associated to each $\eta^{*} \in \mathbb{D}(E^{\sigma})^{*}  $, we view each element $X$ in the Hardy algebra as a $B(H)$-valued function:
%\begin{center}
%$\widehat{X}:\mathbb{D}(E^{\sigma})^{*} \to B(H)$ \\
%\hspace{2.3cm}  $\eta^{*}\to (\sigma \times \eta^{*})(X)$
%\end{center}
%As shown in \cite{Muhly2008b}, each function in $H^{\infty}(E)$ %pg367
%is analytic since it can be represented in terms of a transfer function of a unitary system matrix. These functions are also called Schur class operator functions and may be characterized in terms of certain Pick-like kernels.

The space $E^{\sigma}$ is itself a $W^{*}$-correspondence over $\sigma(A)'$ with the actions given by $a\cdot \eta=(I_{E}\otimes a)\eta$ and $\eta \cdot a=\eta a$ for $\eta \in E$ and $a\in A$.
The $\sigma(A)'$-valued inner product is given by $\langle \eta, \xi \rangle=\eta^{*}\xi$. 

A \emph{dual operator algebra} is an operator algebra $A$ which is also a dual operator space. Any weak$^{*}$-closed subalgebra of $B(H)$ is a dual operator algebra and conversely, for any dual operator algebra $A$, tA \emph{dual operator algebra} is an operator algebra $A$ which is also a dual operator space. Any weak$^{*}$-closed subalgebra of $B(H)$ is a dual operator algebra and conversely, for any dual operator algebra $A$, there is a Hilbert space $H$ and a $w^{*}$-continuous completely isometric homomorphism $\varphi:A\to B(H)$. By the Krein-Smulian theorem, $\varphi(A)$ is a weak$^{*}$-closed subalgebra of $B(H)$, so we can identify $A$ with $\varphi(A)$ as dual operator algebras. A \emph{normal representation} of a dual operator algebra is a completely contractive, $w^{*}$-continuous homomorphism  $\varphi:A\to B(H)$. The \emph{category of normal representations of $A$} is denoted by  $_{A}\mathfrak{M}$. The objects of $_{A}\mathfrak{M}$ are pairs $(H,\varphi)$ where $H$ is a Hilbert space and $\varphi:A\to B(H)$ is a unital completely contractive, $w^{*}$-continuous homomorphism. If $(H_{i},\varphi_{i}), i=1,2,$ are objects in $_{A}\mathfrak{M}$, the morphisms are given by Hom$_{A}(H_{1},H_{2})=\{T\in B(H_{1},H_{2}):T\varphi_{1}(a)=\varphi_{2}(a)T\thinspace ,\thinspace a\in A\}$. 

%A concrete \emph{dual operator $A-B$ bimodule} is a $w^{*}$-closed subspace $X$ of $B(K,H)$ such that $\alpha(A)X\beta(B)\subset X$ where $\alpha$ and $\beta$ are normal representations of $A$ and $B$ on $H$ and $K$ respectively.  An abstract dual operator
If $A$ and $B$ are dual operator algebras, a \emph{dual operator $A-B$-bimodule} is a nondegenerate operator $A-B$-bimodule $X$, which is also a dual operator space, such that the module actions are separately weak$^{*}$-continuous. If $X$ and $Y$ are right operator modules over 
%$A$ and $B$, and $C$ and $B$ respectively
$B$, then we write $CB_{B}(X,Y)$ for the set of completely bounded right $B$-module maps from $X$ to $Y$. If $X$ and $Y$ are left operator modules over 
%$A$ and $B$, and $C$ and $B$ respectively, 
$A$, then we write $_{A}CB(X,Y)$ for the set of completely bounded left $A$-module maps from $X$ to $Y$. Similarly, we write $w^{*}CB_{B}(X,Y)$ for the set of $w^{*}$-continuous completely bounded right $B$-module maps from $X$ to $Y$.

\section{Morita Equivalence of $W^{*}$-correspondences and\\Hardy Algebras} \label{General}
%In 1958, Kiiti Morita developed an equivalence relation for rings \cite{Morita1958}:
%%, which is coarser than that of isomorphism: 
%Two rings are Morita equivalent if their categories of modules are equivalent.
%In 1974, Rieffel introduced a notion of (strong) Morita equivalence for $C^{*}$-algebras and a notion of (weak) Morita equivalence for $W^{*}$-algebras \cite{R1974a}. If $A$ and $B$ are $W^{*}$-algebras, then an $A$-$B$  \emph{$W^{*}$-equivalence bimodule} is an $A$-$B$ $W^{*}$-bimodule $X$ which is a $w^{*}$-full right $W^{*}$-module over $B$ and a $w^{*}$-full left $W^{*}$-module over $A$, such that the two (left and right) inner products of $X$ are compatible in the sense that $_{A}\langle x,y \rangle \cdot z=x\cdot \langle y,z \rangle _{B} $ for all $x,y,z \in X$. If there exists such a bimodule over $A$ and $B$, then we say that $A$ and $B$ are \emph{weakly Morita equivalent}. %What does Morita equivalence do for us?
% Essentially, two Morita equivalent $W^{*}$-algebras have the same representation theory and the same ideal structure. More precisely, the categories of normal representations of both algebras are equivalent and the lattices of weak$^{*}$-closed ideals of both algebras are isomorphic.
In 2000, Muhly and Solel introduced a notion of (strong) Morita equivalence for $C^{*}$-correspondences \cite[Definition 2.1]{Muhly2000a}. This notion can be extended to $W^{*}$-correspondences in the following way: $W^{*}$-correspondences $_{A}E_{A}$ and $_{B}F_{B}$ are called \emph{(weakly) Morita equivalent} if the $W^{*}$-algebras $A$ and $B$ are weakly Morita equivalent via a $W^{*}$-equivalence bimodule $X$ for which there is an $A$-$B$ \thinspace $W^{*}$-correspondence isomorphism $W$ from $X \overbar{\otimes}_{B}F$ onto $E\overbar{\otimes}_{A}X$. %This two tensor products are really the w* Haagerup tensor products by Shur paper comment after definition 2.2 pg 369 and by Blecher1992c, fifth line of pg 126, and by Blecher1992c pg 132 definition on second line of 3. and the fact that the Haagerup tensor product is the same as the interior product (Merdy 8.5.40 and  Blecher-Some General Theory... pg135 after theorem 11.1)
In this case, we will write $E\stackrel{WME}{\sim}_{X}F$. %Merdy pg 336 and Shur paper pg 369
% The modules X tensor_{B} F and E tensor_{A} X here are self dual completions of the balanced C*-tensor products (Morita transforms paper pg 92).-
%Recall that since we are working with $W^{*}$-correspondences,
Recall that $X\overbar{\otimes}_{B}F$ and $E\overbar{\otimes}_{A}X$ are the self dual completions of the balanced $C^{*}$-module interior tensor products. Remember also that the $C^{*}$-module interior tensor product coincides with the module Haagerup tensor product while the $W^{*}$-module tensor product $\overbar{\otimes}_{A}$ coincides with the module weak$^{*}$ Haagerup tensor product $\otimes _{w^{*}h}$, which is the same as the extended Haagerup tensor product $\otimes _{eh}$ \cite [8.5.40, 1.7.1.5]{Blecher2004d}.
Throughout this section, $_{A}E_{A}$ and $_{B}F_{B}$ are $W^{*}$-correspondences over the $W^{*}$-algebras $A$ and $B$.

In \cite{Blecher2000}, Blecher, Muhly and Paulsen generalized Rieffel's strong Morita equivalence of $C^{*}$-algebras (\cite{R1974a}) to general operator algebras. Their generalization is a natural variation of the theory of Morita equivalence that one finds in pure algebra, where the description of Morita equivalence is given in terms of Morita contexts (these contexts are also found in the pure algebra literarture under the name: sets of pre-equivalence data). Their definition is the following: Let $A$ and $B$ be unital or approximately unital operator algebras. Let $X$ be an $A$-$B$ operator bimodule, and let $Y$ a $B$-$A$ operator bimodule. Let $(\cdot , \cdot)$ be a completely bounded bilinear map from $X\times Y$ to $A$, balanced over $B$. Let $[\cdot , \cdot]$ be a completely bounded bilinear map from $Y\times X$ to $B$, balanced over $A$. The 6-tuple $(A,B,X,Y,(\cdot , \cdot),[\cdot , \cdot])$ is called a \emph{(strong) Morita context} for $A$ and $B$ if the module actions are completely contractive and:

\begin{itemize}[leftmargin=*]
\item $(x_{1},y)\cdot x_{2}=x_{1}\cdot [y,x_{2}]\quad x_{1}, x_{2}\in X, y\in Y$.

\NoIndent{$[y_{1},x]\cdot y_{2}=y_{1}\cdot (x,y_{2}) y_{1},\quad y_{2}\in Y, x\in X$.} %this guarantees associativity of the Linking algebra, hence associativity of tensor products of the bimodules.

\item The linear map from $X\otimes _{h}Y$ to $A$ determined by $(\cdot , \cdot)$ is a complete quotient map onto $A$. %this guarantees that X is a generator for the category of right B-Modules

\item The linear map from $Y\otimes _{h}X$ to $B$ determined by $[\cdot , \cdot]$ is a complete quotient map onto $B$.
\end{itemize}
As shown in \cite{Blecher2000}, a (strong) Morita context determines a pair of equivalence functors between the categories of operator modules of both operator algebras in the context. It also determines an equivalence between the categories of Hilbert modules of both operator algebras. Furthermore, the Morita context gives rise to an isomorphism between the lattices of ideals of both operator algebras in the context.
% The typical example of a (strong) Morita context is $(M_{n}(A), A,$ $ C_{n}(A),R_{n}(A) ,\cdot, \cdot)$, where $A$ is an operator algebra, $C_{n}(A)$ and $R_{n}(A)$ are the columns and the rows over $A$, and the last two entries in the context correspond to the usual matrix multiplication of those columns and rows. This context gives the Morita equivalence of $A$ and $M_{n}(A)$. 
One important shortcoming of this notion of strong Morita equivalence is that if the two operator algebras $A$ and $B$ that we are comparing, are dual operator algebras then the strong Morita context does not capture this duality. More precisely, the two functors that are derived from the context, do not give an equivalence between the categories of normal representations of $A$ and $B$.

 In \cite{Eleftherakis2008b}, Eleftherakis formulated a version of Morita theory for dual operator algebras using ternary rings of operators (TROs) and a relation called $\Delta$-equivalence. In \cite{Eleftherakis2008}, Eleftherakis and Paulsen showed that this notion of $\Delta$-equivalence %in \cite{Eleftherakis2008b}
  is equivalent to the notion of weak$^{*}$ stable isomorphism of dual operator algebras. In \cite{Blecher2008a}, Blecher and Kashyap introduced a new notion of weak Morita equivalence of dual operator algebras which includes most of the examples of Morita-like equivalence (in the dual setting) found in the  literature. %A great advantage of this notion is that it gives a framework in which all the important parts of the theory of strong Morita equivalence transfer in a very clean manner. 
  Also, this approach contains the notion of stable isomorphism given by Eleftherakis and Paulsen. That is, if two unital dual operator algebras are weak$^{*}$ stably isomorphic then they are weakly Morita equivalent in the sense of \cite{Blecher2008a}. %As we will see shortly, the notion of weak Morita equivalence is also based on the use of Morita contexts. The difference is that now we will use weak Morita contexts which capture the duality of the algebras involved. More precisely, the two functors derived from a weak Morita context do implement an equivalence between the categories of normal representations of the algebras in the context. Moreover, a weak Morita context gives rise to an isomorphism between the lattices of weak$^{*}$-closed ideals of both dual algebras in the context. That is, the two dual operator algebras in the context have the same representation theory and the same ideal structure.

In the following definitions, $A$ and $B$ are dual operator algebras, $X$ and $Y$ are dual operator bimodules, $X$ is an $A$-$B$ bimodule and $Y$ is an $B$-$A$-bimodule. The following two definitions of Morita equivalence for unital dual operator algebras are found in \cite[section 3]{Blecher2008a}. $A$ and $B$ are called \emph{weak$^{*}$ Morita equivalent}, with equivalence bimodules $X$ and $Y$, if $A\cong X\overbar{\otimes}_{B}Y$ as dual operator $A$-bimodules, and similarly $B \cong Y\overbar{\otimes}_{B}X$ as dual operator $B$-bimodules. 
%, where $\otimes ^{\sigma h}_{B}$ denotes the normal Haagerup tensor product.
 $(A,B,X,Y)$ is called a \emph{weak$^{*}$ Morita Context}. 

For the next definition, assume that $(\cdot , \cdot):X\times Y\to A$ and $[\cdot ,\cdot ]:Y\times X \to B $ are separately weak$^{*}$-continuous completely contractive bilinear maps. We will use the $6$-tuple, or \emph{context,} $(A,B,X,Y,(\cdot , \cdot),[\cdot ,\cdot ])$. We say that $A$ is \emph{weakly Morita equivalent} to $B$ if there exists a 6-tuple as above, there exist $w^{*}$-dense operator algebras $A'$ and $B'$ in $A$ and $B$ respectively, there exists a  $w^{*}$-dense operator $A'$-$B'$-submodule $X'$ in $X$, and a  $w^{*}$-dense operator $B'$-$A'$-submodule $Y'$ in $Y$, such that the subcontext $(A',B',X',Y')$, together with restrictions of the maps $(\cdot , \cdot)$ and $[\cdot ,\cdot ]$, is a strong Morita context in the sense of \cite[Definition 3.1]{Blecher2000}. $(A,B,X,Y,(\cdot , \cdot),[\cdot ,\cdot ])$ is called a \emph{weak Morita Context}. %Weak Morita equivalence is nothing more than the weak$^{*}$-closure of a strong Morita equivalence in the sense of \cite{Blecher2000}.
 
%In \cite[Theorems 3.5, 3.10 and 4.3]{Blecher2008a}, it is shown that weak$^{*}$ Morita equivalent dual operator algebras have equivalent categories of dual operator modules and equivalent categories of normal Hilbert modules, via the functors $\mathcal{F}$ and $\mathcal{G}$ given by $\mathcal{F}(Z)=X\overbar{\otimes}_{B}Z$ and $\mathcal{G}(W)=Y\overbar{\otimes}_{A}W$. Furthermore, if $(A,B,X,Y)$ is a weak$^{*}$ Morita context, then $A$ and $B$ have isomorphic lattices of weak$^{*}$-closed ideals.
 
Our goal in this section is to show that if two $W^{*}$-correspondences $(E,A)$ and $(F,B)$ are (weakly) Morita equivalent then their Hardy algebras $H^{\infty} (E)$ and $H^{\infty} (F)$ are weakly Morita equivalent and weak$^{*}$ Morita equivalent.
\begin{lemma} \label{mA}
Let $X$ be  an $A$-$B$ $W^{*}$-equivalence bimodule. Then the pairs $(I_{A}, m_{A})$ and $(I_{B}, m_{B})$, where $m_{A}:X\overbar{\otimes}_{B} \widetilde{X}\to A $ and $m_{B}:\widetilde{X}\overbar{\otimes}_{A}  X \to B$ are defined by $m_{A}(x\otimes \widetilde{y})=_{A}\langle x,y \rangle$ and $m_{B}(\widetilde{x} \otimes y)=\langle x,y \rangle_{B}$, are $W^{*}$-correspondence isomorphisms.
\end{lemma}
\begin{proof}
The two identities are obviously $W^{*}$-algebra isomorphisms.
\begin{align*}% left aligned
m_{A}(a\cdot (x\otimes \widetilde{y})\cdot b)&=m_{A}((a\cdot x)\otimes (\widetilde{y}\cdot b))=m_{A}((a\cdot x)\otimes (b^{*} \cdot \widetilde{y}))\\
&=_{A}\langle a\cdot x,b^{*}\cdot y \rangle =a _{A}\langle x, y \rangle b \\
&=am_{A}(x\otimes y)b
\end{align*}
\begin{align*}% left aligned
\langle m_{A}(x\otimes \widetilde{y}),m_{A}(z\otimes \widetilde{w})\rangle _{A}&=\langle _{A}\langle x, y \rangle, _{A}\langle z, w \rangle \rangle _{A}= _{A}\langle x, y \rangle^{*}\cdot_{A}\langle z, w \rangle =_{A}\langle y, x \rangle \cdot_{A}\langle z, w \rangle\\
&=_{A}\langle _{A} \langle y, x \rangle \cdot z, w\rangle =_{A}\langle y \cdot \langle x, z \rangle_{B}, w\rangle =_{A}\langle y, w\cdot \langle x, z \rangle_{B}^{*}\rangle \\
&=\langle \widetilde{y},\widetilde{w\cdot \langle x, z \rangle_{B}^{*}}\rangle_{A}=\langle \widetilde{y},\langle x, z \rangle_{B} \cdot \widetilde{w}\rangle_{A}=\langle x\otimes \widetilde{y},z\otimes \widetilde{w}\rangle _{A}&
\end{align*}
\begin{align*}% left aligned
_{A}\langle m_{A}(x\otimes \widetilde{y}),m_{A}(z\otimes \widetilde{w})\rangle &=_{A}\langle _{A}\langle x, y \rangle, _{A}\langle z, w \rangle \rangle = _{A}\langle x, y \rangle\cdot_{A}\langle z, w \rangle^{*} =_{A}\langle x, y \rangle \cdot_{A}\langle w, z \rangle\\
&=_{A}\langle _{A} \langle x, y \rangle \cdot w, z\rangle =_{A}\langle x \cdot \langle y, w \rangle_{B}, z \rangle =_{A}\langle x, _{B}\langle \widetilde{y},\widetilde{w} \rangle, z\rangle \\
&=_{A} \langle x\otimes \widetilde{y},z\otimes \widetilde{w}\rangle &
\end{align*}
That is, $m_{A}$ preserves both left and right inner products. So it is isometric, hence injective with closed range. Since $m_{A}$ is defined in terms of the left inner product of $X$, and by definition, $X$ is a $w^{*}$-full left $W^{*}$-module over $A$, $m_{A}$ has $w^{*}$-dense range in $A$. So $m_{A}$ is surjective, thus $(I_{A}, m_{A})$ is a $W^{*}$-correspondence isomorphism. Similarly, $(I_{B}, m_{B})$ is also a $W^{*}$-correspondence isomorphism.
\end{proof}
\begin{theorem}\label{Moritahardy}
If two $W^{*}$-correspondences $(E,A)$ and $(F,B)$ are (weakly) Morita equivalent then their Hardy algebras $H^{\infty} (E)$ and $H^{\infty} (F)$ are weakly Morita equivalent and weak$^{*}$ Morita equivalent (as dual operator algebras) in the sense of \cite{Blecher2008a}.
\end{theorem}
\begin{proof}
We model our proof on the proof given in \cite{Muhly2000a} for the $C^{*}$ case. $A$ and $B$ are weakly Morita equivalent $W^{*}$-algebras via a $W^{*}$-equivalence bimodule $_{A}X_{B}$ and there is a $W^{*}$-correspondence isomorphism $W: X \overbar{\otimes}_{B}F \to E\overbar{\otimes}_{A}X$. Let $\mathcal{I}$ (and  $\mathcal{I}^{w} $) denote the norm closure in $A$ (and the $w^{*}$-closure in $A$) of the span of the range of the $A$-valued inner product on $X$. Form the linking $W^{*}$-algebra $L$ for $\widetilde{X}$. That is, let
\begin{equation*}
L = 
\left( 
\begin{matrix} 
B & \widetilde{X} \\
X & A
\end{matrix}
\right) .
\end{equation*}
%Since $\widetilde{X}\oplus ^{c} A$ is a right $W^{*}$-module over $A$ (as shown in \cite[section 4]{Blecher2008a}), $L=B(\widetilde{X}\oplus ^{c} A)=\ms{L} (\widetilde{X}\oplus ^{c} A)$ is a $W^{*}$-algebra.
 % Merdy 8.5.10 and 8.5.1 (2)
 
Let $Y_{1}$ be the first column, $
\left( 
\begin{matrix} 
B  \\
X 
\end{matrix}
\right) $. This is a $W^{*}$-module over $B$ since it is the column sum of $B$ and $X$. Likewise, the second column, $Y_{2}=
\left( 
\begin{matrix} 
\widetilde{X}  \\
A 
\end{matrix}
\right) $, is a $W^{*}$-module over $A$. $CB_{B}(X)=B_{B}(X)=\ms{L}_{B}(X)= M(K_{B}(X))\cong M(\mathcal{I})= \mathcal{I} ^{w} = A$ %by first line of proof of Merdy 8.5.5, 8.5.5(1), Lance theorem 2.4, Merdy 8.1.15 (I=K(X), proof of 8.1.15 bottom of pg 302 Merdy), 8.5.3 and 8.5.1.
%$CB_{B}(X)=\ms{L}_{B}(X)=B_{B}(X)=A$ %  by Merdy 8.5.13, 8.5.5 (1) and first line of proof of 8.5.5, A=B(X)=CB(X).
and %$CB_{B}(B)=B_{B}(B)= \ms{L}_{B}(B)= M(K_{B}(B))= \mathcal{I} ^{w}=B$.
$CB_{B}(B)=B$. % by Merdy 3.5.4 applied to B as a right module over itself.
 Since $X$ is a selfdual space, $CB_{B}(B,X)\cong X$ and $CB_{B}(X,B)\cong \widetilde{X}$. So we have that $L \cong CB_{B}(Y_{1})$. Similarly,  $CB_{A}(\widetilde{X})\cong B$, $CB_{A}(A,\widetilde{X})\cong \widetilde{X} $, $CB_{A}(\widetilde{X},A)\cong X$ and $CB_{A}(A)\cong A$. So $L \cong CB_{A}(Y_{2})$. That is, $CB_{B}(Y_{1})$ and $CB_{A}(Y_{2})$ are $W^{*}$-algebras. For more information on the previous identifications, see \cite[8.5.5, 8.1.15, 8.5.3, 8.5.1, 8.5.13, 8.5.5 (1), 2.6.1, 3.5.4 (2)]{Blecher2004d} for example.
 % \cite[8.5.13, 8.5.1(2), 8.5.5 (1), 2.6.1, 3.5.4 (2), 8.5.1 (1)]{Blecher2004d}. 
%It might be also useful to recall that by Lance-Theorem 2.4, adjointable(X)=M(K(X)).

Now we have that $CB_{B}(Y_{1},F)$ is a module over the $W^{*}$-algebra $CB_{B}(Y_{1})=L$, where right multiplication is given by composition and the inner product is given by $\langle T,S \rangle= T^{*}S \in CB_{B}(Y_{1})=L$. By \cite[equation $(\dagger \dagger)$]{Blecher1997a}, $CB_{B}(Y_{1},F)=(Y_{1}\hat{\otimes} _{B}F_{*})^{*}$, where 
 $\hat{\otimes} _{B}$ is the module operator space projective tensor product. % By \cite[equation $(\dagger \dagger)$]{Blecher1997a}, $CB_{B}(Y_{1},F)=(Y \wideparen{\otimes} _{B}F_{*})^{*}$, where $\wideparen{\otimes} _{B}$ is the module operator space projective tensor product.
So $CB_{B}(Y_{1},F)$ is a $W^{*}$-module over $CB_{B}(Y_{1})=L$ (by \cite[corollary 8.5.7]{Blecher2004d}). Similarly, $CB_{A}(Y_{2},E)$ is a $W^{*}$-module over $CB_{A}(Y_{2})=L$. So their sum    % need to show that $CB_{B}(Y_{1},F)$ is self dual = it has a predual. For the initial identifications, I have used Merdy 3.5.4(2), 8.5.1(1) and 8.5.1(2). For an algebra A, B_{A}(A)=CB_{A}(A) by Merdy pg 82. Maybe it is also possible to use:If a C^* algebra A is unital then K(A)=\scr{L}(A) (Lance pg 10).
\begin{align*}
Z=CB_{B}(Y_{1},F)\oplus CB_{A}(Y_{2},E)
\end{align*}
is a  $W^{*}$-module over $L$.  % Merdy 8.5.26
Since $CB_{B}(Y_{1},F)$ can be written as $(CB_{B}(B,F),CB_{B}(X,F))$ and $CB_{A}(Y_{2},E)$ is $(CB_{A}(\widetilde{X},E),CB_{A}(A,E))$, we have
\begin{equation*}
Z = 
\left( 
\begin{matrix} 
CB_{B}(B,F) & CB_{B}(X,F) \\
CB_{A}(\widetilde{X},E) & CB_{A}(A,E)
\end{matrix}
\right) 
\end{equation*}
The right action of $L$ on $Z$ is realized as the usual matrix multiplication, and the inner product is given by 
\begin{align*}
\langle 
\left( 
\begin{matrix} 
T_{11} & T_{12} \\
T_{21} & T_{22}
\end{matrix}
\right),
 \left( 
\begin{matrix} 
S_{11} & S_{12} \\
S_{21} & S_{22}
\end{matrix}
\right ) \rangle =
\left( 
\begin{matrix} 
T_{11}^{*} & T_{12}^{*} \\
T_{21}^{*} & T_{22}^{*}
\end{matrix}
\right)
 \left( 
\begin{matrix} 
S_{11} & S_{12} \\
S_{21} & S_{22}
\end{matrix}
\right ) \in L
\end{align*}
Since $CB_{B}(B,F)\cong F$, $CB_{B}(X,F)\cong F\otimes \widetilde{X}$,  $CB_{A}(\widetilde{X},E)\cong E\otimes X$ and $_{A}CB(A,E) \cong E$, we have %by Th 2.5 in Blecher1997a
\begin{align*}
Z = 
\left( 
\begin{matrix} 
F & F\overbar{\otimes}_{B} \widetilde{X} \\
E \overbar{\otimes}_{A} X & E
\end{matrix}
\right) 
\end{align*}
The right action then becomes 
\begin{align*} 
\left( 
\begin{matrix} 
f_{1} & f_{2} \otimes  \widetilde{x}  \\
e_{1} \otimes y & e_{2}
\end{matrix}
\right)&
 \left( 
\begin{matrix} 
b & z \\
u & a
\end{matrix}
\right )\\
& =
\left( 
\begin{matrix} 
f_{1}b+I_{F} \otimes m_{B}(f_{2} \otimes \widetilde{x} \otimes u) & f_{1} \otimes \widetilde{z}+f_{2} \otimes \widetilde{x}a \\
e_{1} \otimes yb+e_{2} \otimes u & I_{E} \otimes m_{A}(e_{1} \otimes y \otimes  \widetilde{z})+e_{2}a
\end{matrix}
\right )
\end{align*}
and the inner product is

\begin{align*}
\langle 
\left( 
\begin{matrix} 
f_{1} & f_{2} \otimes  \widetilde{x}  \\
e_{1} \otimes y & e_{2}
\end{matrix}
\right)&,
\left( 
\begin{matrix} 
g_{1} & g_{2} \otimes  \widetilde{z}  \\
k_{1} \otimes u & k_{2}
\end{matrix}
\right)
\rangle \\
& =
\left( 
\begin{matrix} 
\langle f_{1},g_{1} \rangle + \langle e_{1} \otimes y,k_{1} \otimes u \rangle _{B}& \langle f_{1},g_{2} \rangle \widetilde{z} +\widetilde{y} \langle e_{1} \otimes k_{2}\rangle \\
x\langle f_{2},g_{1} \rangle + \langle e_{2},k_{1}\rangle u & \langle f_{2} \otimes \widetilde{x},g_{2} \otimes \widetilde{z} \rangle _{A} +\langle e_{2}, k_{2}\rangle
\end{matrix}
\right).
\end{align*} 
Let $\varphi_{Z}:L\to \ms{L}(Z)$ be defined by
\begin{eqnarray*}
\varphi_{Z}(
\left( 
\begin{matrix} 
b & 0 \\
0 & a
\end{matrix}
\right))
\left( 
\begin{matrix} 
f_{1} & f_{2} \otimes  \widetilde{z}  \\
e_{1} \otimes v & e_{2}
\end{matrix}
\right) &=& \left( 
\begin{matrix} 
\varphi_{F}(b)f_{1} & \varphi_{F}(b) f_{2} \otimes  \widetilde{z}  \\
\varphi_{E}(a)e_{1} \otimes v & \varphi_{E}(a)e_{2}
\end{matrix}
\right),\\
\varphi_{Z}(
\left( 
\begin{matrix} 
0 & 0 \\
x & 0
\end{matrix}
\right))
\left( 
\begin{matrix} 
f_{1} & f_{2} \otimes  \widetilde{z}  \\
e_{1} \otimes v & e_{2}
\end{matrix}
\right)  &=& \left( 
\begin{matrix} 
0 & 0  \\
W(x \otimes f_{1}) & (I_{E} \otimes m_{A})(W \otimes I_{\widetilde{X}})(x \otimes f_{2} \otimes \widetilde{z})
\end{matrix}
\right),\\
\varphi_{Z}(
\left( 
\begin{matrix} 
0 & \widetilde{y} \\
0 & 0
\end{matrix}
\right))
\left( 
\begin{matrix} 
f_{1} & f_{2} \otimes  \widetilde{z}  \\
e_{1} \otimes v & e_{2}
\end{matrix}
\right)  &=& \left( 
\begin{matrix} 
(m_{B} \otimes I_{F})(I_{\widetilde{X}} \otimes W^{-1})(\widetilde{y} \otimes e_{1} \otimes v) & \widetilde{W}(\widetilde{y} \otimes e_{1}) \\
0 & 0
\end{matrix}
\right),
\end{eqnarray*}
where $\widetilde{W}:\widetilde{X} \overbar{\otimes} _{A}E \to F \overbar{\otimes} _{B} \widetilde{X}$ is the isomorphism defined in \cite[Section 2, pg 116]{Muhly2000a}. By \cite[Proposition 2.6]{Muhly2000a}, $\varphi_{Z}$ is a $*$-homomorphism, % (I need to show $\varphi_{Z}$ is also normal).
and since $W$ is $w^{*}$-continuous (being a $W^{*}$-correspondence isomorphism), $\varphi_{Z}$ is normal. Thus  $Z$ is a  $W^{*}$-correspondence over $L$ (where the left action of $L$ on $Z$ is given by $\varphi_{Z}$).

Replacing the $C^{*}$-module interior tensor product (which is the same as the module Haagerup tensor product) with the $W^{*}$-module tensor product, and replacing the direct sum with the ultraweak direct sum in \cite[Lemmas 2.7, 2.8 and 2.10]{Muhly2000a}, %$\ms{F}(E) \stackrel{WME}{\sim}_{X} \ms{F}(F)$, where the associated correspondence isomorphism is the direct sum of the maps $W_{k}:X\otimes _{B} F^{\otimes ^{k}} \to  E^{\otimes ^{k}} \otimes _{A} X$ defined by $W_{k}=(I_{E} \otimes W_{k-1})(W_{1}\otimes I_{F^{\otimes(k-1)}})$. Furthermore,
we have that the Fock space $\ms{F}(Z)$ can be written in the form
\begin{align*}
\mathscr{F}(Z)\cong
\left( 
\begin{matrix} 
\mathscr{F}(F) & \mathscr{F}(F) \overbar{\otimes}_{B}  \widetilde{X} \\
\mathscr{F}(E)\overbar{\otimes}_{A} X & \mathscr{F}(E)
\end{matrix}
\right).
\end{align*}
%In what follows, we assume that our $W^{*}$-algebras are weakly Morita equivalent via a fixed $W^{*}$-equivalence bimodule $X$. We also assume that our $W^{*}$-correspondences $E$ and $F$ are (weakly) Morita equivalent via $X$. We write $W$ and $\widetilde{W}$ for the associated isometries and we write $Z$ for the associated $W^{*}$-correspondence over $L$.

Form the operator algebras $\mathcal{T}(Z),\mathcal{T}_{+}(Z)$ and $H^{\infty}(Z)$ associated with $Z$. Consider the submodule
\begin{align*}
\ms{F}'(Z)=
\left( 
\begin{matrix} 
\ms{F}(F) & 0 \\
\ms{F}(E)\overbar{\otimes}_{A}X & 0
\end{matrix}
\right) \subset \ms{F}(Z).
\end{align*}

By the definition of the map $\varphi_{Z}$,  $\ms{F}'(Z)$ is invariant for the diagonal operators in $\ms{L}(\ms{F}(Z))$. By \cite[Lemma 2.9]{Muhly2000a},  $\ms{F}'(Z)$ is also invariant for the creation operators in $\ms{L}(\ms{F}(Z))$. Thus, $\ms{F}'(Z)$ is invariant for $H^{\infty}(Z)$. %Since a similar argument shows that the submodule 
%\begin{align*}
%\ms{F}''(Z)=
%\left( 
%\begin{matrix} 
%0 & \ms{F}(F) \otimes \widetilde{X} \\
%0 & \ms{F}(E)
%\end{matrix}
%\right)
%\end{align*}
%is also invariant for $\mathcal{T}_{+}(Z)$ and $H^{\infty}(Z)$, the two submodules reduce $H^{\infty}(Z)$.
Furthermore, by \cite[Lemma 3.1]{Muhly2000a}, the representation of $\mathcal{T}(Z)$ obtained by restricting the action of $\mathcal{T}(Z)$ to $\ms{F}'(Z)$ %and to $\ms{F}''(Z)$ are
is faithful. That is, we can study the action of $\mathcal{T}(Z)$ on $\ms{F}(Z)$ by just studying
 %looking
  the action of $\mathcal{T}(Z)$ on $\ms{F}'(Z)$.

Write $p$ for the projection in $\ms{L}(\ms{F}'(Z))$ onto $
\left( 
\begin{matrix} 
\ms{F}(F)  \\
0
\end{matrix}
\right) $ and $q$ for the projection onto  $
\left( 
\begin{matrix} 
0  \\
\ms{F}(E)\overbar{\otimes} _{A}X
\end{matrix}
\right) $.
Next, we show that $pH^{\infty}(Z)p\cong H^{\infty}(F)$ and $qH^{\infty}(Z)q\cong H^{\infty}(E)$. Let 
\begin{align*}
\xi=
\left( 
\begin{matrix} 
h_{1} & h_{2}\otimes \widetilde{w} \\
k_{1} \otimes v & k_{2}
\end{matrix}
\right) \in Z^{\otimes m}
\end{align*}
and let $f \in F^{\otimes l}$. We can view $f$ as the element 
$\left( 
\begin{matrix} 
f \\
0
\end{matrix}
\right) \in 
\left( 
\begin{matrix} 
F^{\otimes l}  \\
0
\end{matrix}
\right) \subset \ms{F}'(Z)$. By \cite[Lemma 2.9]{Muhly2000a},
\begin{align*}
T_{\xi}p
\left( 
\begin{matrix} 
f \\
0
\end{matrix}
\right)= 
T_{\xi}
\left( 
\begin{matrix} 
f \\
0
\end{matrix}
\right)= 
\xi \otimes
\left( 
\begin{matrix} 
f & 0 \\
0 & 0
\end{matrix}
\right)=
\left( 
\begin{matrix} 
h_{1} \otimes f & 0 \\
k_{1} \otimes W_{l}(v\otimes f) & 0
\end{matrix}
\right),
\end{align*}
where $W_{k}=(I_{E} \otimes W_{k-1})(W_{1}\otimes I_{F^{\otimes(k-1)}})$. % bottom of pg 134
 Hence
\begin{align*}
pT_{\xi}p
\left( 
\begin{matrix} 
f \\
0
\end{matrix}
\right)= 
\left( 
\begin{matrix} 
h_{1} \otimes f \\
0
\end{matrix}
\right),
\end{align*}
which we write as $pT_{\xi}p=T_{h_{1}}$. For
\begin{align*}
\lambda= 
\left( 
\begin{matrix} 
b & \widetilde{y} \\
x & a
\end{matrix}
\right) \in L,
\end{align*}
we have 
\begin{align*}
\varphi _{\infty}(\lambda) p
\left( 
\begin{matrix} 
f & 0 \\
0 & 0
\end{matrix}
\right)=
\varphi _{\infty}(\lambda) 
\left( 
\begin{matrix} 
f & 0 \\
0 & 0
\end{matrix}
\right)=
\left( 
\begin{matrix} 
\varphi _{F^{\otimes l}}(b)f & 0 \\
W_{l}(x \otimes f) & 0
\end{matrix}
\right).
\end{align*}
Hence
\begin{align*}
p\varphi _{\infty}(\lambda) p
\left( 
\begin{matrix} 
f & 0 \\
0 & 0
\end{matrix}
\right)=
\left( 
\begin{matrix} 
\varphi _{F^{\otimes l}}(b)f & 0 \\
0 & 0
\end{matrix}
\right),
\end{align*}
So $p\varphi _{\infty}(\lambda)p=\varphi_{\infty}(b)$. That is, the generators of the algebra $p\mathcal{T}_{+}(Z)p$ are identified with the generators of $\mathcal{T}_{+}(F)$. Thus $p\mathcal{T}_{+}(Z)p\cong \mathcal{T}_{+}(F)$ and  $pH^{\infty}(Z)p\cong H^{\infty}(F)$.

Similarly, we can view the element $e \otimes u \in E^{\otimes l} \otimes X$ as the element $
\left( 
\begin{matrix} 
0 & 0 \\
e \otimes u & 0
\end{matrix}
\right) \in \ms{F}'(Z)$. So by \cite[Lemma 2.9]{Muhly2000a},
\begin{align*}
T_{\xi}q
\left( 
\begin{matrix} 
0 & 0 \\
e \otimes u & 0
\end{matrix}
\right)=
T_{\xi}
\left( 
\begin{matrix} 
0 & 0 \\
e \otimes u & 0
\end{matrix}
\right)=
\xi \otimes
\left( 
\begin{matrix} 
0 & 0 \\
e \otimes u & 0
\end{matrix}
\right)=
\left( 
\begin{matrix} 
c(h_{2}\otimes \widetilde{w}, W_{l}^{-1}(e\otimes u)) & 0 \\
k_{2}\otimes e\otimes u & 0
\end{matrix}
\right),
\end{align*} where $c:F^{\otimes m} \otimes \widetilde{X} \times X \otimes F^{\otimes l} \to F^{\otimes m} \otimes F^{\otimes l}$ is a bilinear map which is not relevant for our purposes.
So
\begin{align*}
qT_{\xi}q
\left( 
\begin{matrix} 
0 & 0 \\
e \otimes u & 0
\end{matrix}
\right)=
\left( 
\begin{matrix} 
0 & 0 \\
k_{2}\otimes e\otimes u & 0
\end{matrix}
\right),
\end{align*}
which we write as $qT_{\xi}q=T_{k_{2}} \otimes I_{x}$. For
\begin{align*}
\lambda= 
\left( 
\begin{matrix} 
b & \widetilde{y} \\
x & a
\end{matrix}
\right) \in L,
\end{align*}
we have 
\begin{align*}
\varphi _{\infty}(\lambda)
\left( 
\begin{matrix} 
0 & 0 \\
e \otimes u & 0
\end{matrix}
\right)=
\left( 
\begin{matrix} 
(m_{B} \otimes I_{F^{\otimes l}}(I_{\widetilde{X}} \otimes W^{-1})(\widetilde{y}\otimes e\otimes u) & 0 \\
\varphi_{E^{\otimes l}}(a)\otimes e\otimes u & 0
\end{matrix}
\right).
\end{align*}
Hence
\begin{align*}
q\varphi _{\infty}(\lambda)q
\left( 
\begin{matrix} 
0 & 0 \\
e \otimes u & 0
\end{matrix}
\right)=
\left( 
\begin{matrix} 
0 & 0 \\
\varphi_{E^{\otimes l}}(a)\otimes e\otimes u & 0
\end{matrix}
\right).
\end{align*}
So $q\varphi _{\infty}(\lambda)q=\varphi_{\infty}(a)\otimes I_{X}$. Since the map from $\ms{L}(\ms{F}(E)\overbar{\otimes} _{A}X)$ to $\ms{L}(\ms{F}(E))$ taking $T \otimes I_{X}$ to $T$ is an isomorphism (\cite[Lemma 2.12]{Muhly2000a}), we have that $q\mathcal{T}_{+}(Z)q\cong \mathcal{T}_{+}(E)$ and  $qH^{\infty}(Z)q\cong H^{\infty}(E)$.

Next, we show that $H^{\infty} (E)$ and $H^{\infty} (F)$ are weakly Morita equivalent (as dual operator algebras) in the sense of \cite{Blecher2008a}. First note that $pH^{\infty}(Z)p$ and $qH^{\infty}(Z)q$ are unital dual operator algebras with identities $p\varphi_{\infty}(1_{L})p= \varphi_{\infty}(1_{B}) $ and $q\varphi_{\infty}(1_{L})q=\varphi_{\infty}(1_{A})$ respectively.

Let $(\cdot,\cdot):pH^{\infty}(Z)q \times qH^{\infty}(Z)p \longrightarrow pH^{\infty}(Z)p$ and $[\cdot,\cdot]:qH^{\infty}(Z)p\times pH^{\infty}(Z)q$ $\longrightarrow qH^{\infty}(Z)q$  be given by:                               
\begin{align*}
(p\alpha q,q\beta p) \longrightarrow p\alpha q\beta p \\
[q\alpha p,p\beta q] \longrightarrow q\alpha p\beta q
\end{align*}
respectively, and let $(\cdot,\cdot)_{t}:p\mathcal{T}_{+}(Z)q$ $\times$ $q\mathcal{T}_{+}(Z)p \longrightarrow p\mathcal{T}_{+}(Z)p$ and $[\cdot,\cdot]_{t}:q\mathcal{T}_{+}(Z)p$ $\times$ $p\mathcal{T}_{+}(Z)q \longrightarrow q\mathcal{T}_{+}(Z)q$ be the respective restrictions of $(\cdot,\cdot)$ and $[\cdot,\cdot]$.

As shown in \cite[Theorem 3.2 (3)]{Muhly2000a},
\begin{align*}
(p\mathcal{T}_{+}(Z)p,q\mathcal{T}_{+}(Z)q,p\mathcal{T}_{+}(Z)q,q\mathcal{T}_{+}(Z)p,(\cdot,\cdot)_{t},[\cdot,\cdot]_{t})
\end{align*}
is a strong Morita context in the sense of \cite[Definition 3.1]{Blecher2000}. In particular, the multiplication maps $(\cdot,\cdot)$ and $[\cdot,\cdot]$ are completely contractive bilinear maps.
% the contractive part follows by Blecher-Ruan-Sinclair theorem  Merdy 2.3.2
 Since for any Hilbert space $H$, the product in $B(H)$ is separately weak$^{*}$-continuous, we have that $(\cdot,\cdot)$ and $[\cdot,\cdot]$ are separately weak$^{*}$-continuous (here we are using the identification of $H^{\infty}(Z)$, an abstract dual operator algebra, with a subalgebra of B(H) via a complete isometric homomorphism which is a $w^{*}$-homeomorphism). % Merdy pg 88 or Merdy characterization of dual operator algebras Theorem 1, and apparently Blecher "The standard dual of an operator space". I still should show that the mutiplication maps ( ) and [ ] are completely contractive on the w* completions.

Then, since $\overline{p\mathcal{T}_{+}(Z)p}^{w^{*}}=pH^{\infty}(Z)p$ , $\overline{q\mathcal{T}_{+}(Z)q}^{w^{*}}=qH^{\infty}(Z)q$, $\overline{p\mathcal{T}_{+}(Z)q}^{w^{*}}=pH^{\infty}(Z)q$ and $\overline{q\mathcal{T}_{+}(Z)p}^{w^{*}}=qH^{\infty}(Z)p$, we have that 
\begin{align*}
(pH^{\infty}(Z)p,qH^{\infty}q,pH^{\infty}(Z)q,qH^{\infty}(Z)p,(\cdot,\cdot),[\cdot,\cdot])
\end{align*}
is a weak Morita context. Thus $H^{\infty}(E) \cong qH^{\infty}(Z)q$ and $H^{\infty}(F)\cong pH^{\infty}(Z)p$ are weakly Morita equivalent in the sense of \cite[Definition 3.2]{Blecher2008a}. By \cite[Corollary 3.4]{Blecher2008a},  $H^{\infty}(E)$ and $H^{\infty}(F)$ are also weak$^{*}$ Morita equivalent.
\end{proof}
In the same way that a strong Morita context gives rise to a linking algebra $\mathcal{L}$ (see \cite[Chapter 3]{Blecher2000}), a weak$^{*}$  Morita context (and therefore a weak Morita context also) gives rise to a weak linking algebra $\mathcal{L}^{\omega}$, which is a dual operator algebra. The construction of $\mathcal{L}^{\omega}$ is given in \cite[Section 4]{Blecher2008a}. In our case, the linking operator algebra of the strong Morita context $(\mathcal{T}_{+}(F),\mathcal{T}_{+}(E),p\mathcal{T}_{+}(Z)q,q\mathcal{T}_{+}(Z)p)$ is 
\begin{equation*}
\mathcal{L}=
\left( 
\begin{matrix} 
\mathcal{T}_{+}(F) & p\mathcal{T}_{+}(Z)q \\
q\mathcal{T}_{+}(Z)p & \mathcal{T}_{+}(E)
\end{matrix}
\right)
\end{equation*}
and the weak linking algebra of the weak Morita context $(H^{\infty}(F),H^{\infty}(E),pH^{\infty}(Z)q,$ $qH^{\infty}(Z)p)$ is
\begin{equation*}
\mathcal{L}^{\omega}=
\left( 
\begin{matrix} 
H^{\infty}(F) & pH^{\infty}(Z)q \\
qH^{\infty}(Z)p & H^{\infty}(E)
\end{matrix}
\right).
\end{equation*} 
$\mathcal{L}$ can be identified completely isometrically with a weak$^{*}$-dense subalgebra of $\mathcal{L}^{\omega}$. 
%Using the notation $\oplus^{c}$ and $\oplus^{r}$ on \cite[pg 48]{Kashyap2008}, and following \cite[4]{Blecher2008a}, since 
Adapting \cite[4]{Blecher2008a} to our algebras, we have that $(H^{\infty}(F),\mathcal{L}^{\omega},H^{\infty}(F)\oplus^{r}pH^{\infty}(Z)q,$ $H^{\infty}(F)\oplus^{c}qH^{\infty}(Z)p)$ and $(H^{\infty}(E),\mathcal{L}^{\omega},H^{\infty}(E)\oplus^{r}qH^{\infty}(Z)p,H^{\infty}(E)\oplus^{c}pH^{\infty}(Z)q)$ are weak$^{*}$ Morita contexts. The next corollary follows.
\begin{corollary}
If two $W^{*}$-correspondences $(E,A)$ and $(F,B)$ are weakly Morita equivalent, then $H^{\infty}(F)$ and $H^{\infty}(E)$ are weakly Morita equivalent to $\mathcal{L}^{\omega}$.
\end{corollary}% from Blecher2008a, corollary 4.1
Furthermore, applying the map $L_{N}$ in \cite[Theorem 3.6]{Blecher2008a} to the weak linking algebras of the weak$^{*}$ Morita contexts $(H^{\infty}(F),\mathcal{L}^{\omega},H^{\infty}(F)\oplus^{r}pH^{\infty}(Z)q,H^{\infty}(F)\oplus^{c}qH^{\infty}(Z)p)$ and $(H^{\infty}(E),\mathcal{L}^{\omega},H^{\infty}(E)\oplus^{r}qH^{\infty}(Z)p,H^{\infty}(E)\oplus^{c}pH^{\infty}(Z)q)$, we have:
\begin{corollary}
If two $W^{*}$-correspondences $(E,A)$ and $(F,B)$ are weakly Morita equivalent, 
%so that $\mathcal{L}^{\omega}$ is the weak linking algebra of $H^{\infty}(F)$ and $H^{\infty}(E)$ 
then $w^{*}CB_{H^{\infty}(F)}(H^{\infty}(F)\oplus^{c}qH^{\infty}(Z)p)\cong \mathcal{L}^{\omega}\cong w^{*}CB_{H^{\infty}(E)}(H^{\infty}(E)\oplus^{c}pH^{\infty}(Z)q)$ completely isometrically and $w^{*}$-isomorphically.
\end{corollary}
In \cite{Blecher1999a}, it was shown that a strong Morita equivalence of operator algebras gives a subcontext of a strong Morita equivalence (in the sense of Rieffel) of containing $C^{*}$-algebras.
% More precisely, if $(A,B,X,Y)$ is a strong Morita context of operator algebras $A$ and $B$, then any $C^{*}$-algebra $C$ generated by (a completely isometric representation of) $A$ induces a $C^{*}$-algebra $D$ generated by (a completely isometric representation of) $B$, such that $C$ and $D$ are stongly Morita equivalent, via the equivalence bimodule $C\otimes_{hA}X$. Furthermore, the linking $C^{*}$-algebra $\mathcal{L}'$ for $C$ and $D$ is generated by (a completely isometric copy of )the linking algebra $\mathcal{L}$ for $A$ and $B$. 
 In \cite{Blecher2008a} and \cite{Kashyap2008}, Blecher and Kashyap, presented an extension of all this theory to the setting of dual operator algebras and weak$^{*}$ Morita equivalence. 
% Their extension is very natural, in the sense that each object and argument in the norm closed case, is replaced by the "expected" weak$^{*}$ substitute. For example, if $A$ is a dual operator algebra, a $W^{*}$-cover %Upasana pg 14 
% of $A$ is a pair $(C,j)$ consisting of a  $W^{*}$-algebra $C$ and a completely isometric $w^{*}$-continuous homomorphism $j:A\to C$, such that $j(A)$ generates $C$ as a $W^{*}$-algebra. By the Krein-Smulian theorem, $j(A)$ is a $w^{*}$-closed subalgebra of $C$. The $W^{*}$-covers of $A$ have a natural order: %Upasana pg 100
%If $(C,j)$ and $(C',j')$ are $W^{*}$-covers of $A$, then we say  $(C,j)\leq(C',j')$ if and only if there is a $w^{*}$-continuous homomorphism $\pi:C'\to C$ such that $\pi \circ j'=j$. Such $\pi$ is surjective. 
Basically, if $(A,B,X,Y)$ is a weak$^{*}$ Morita context of dual operator algebras $A$ and $B$, then any $W^{*}$-algebra $C$ generated by $A$ induces a $W^{*}$-algebra $D$ generated by $B$, such that $C$ and $D$ are Morita equivalent, via the $W^{*}$-equivalence bimodule $C\otimes^{\sigma h}_{A}X$.
 %Furthermore, the linking $W^{*}$-algebra $\mathcal{L^{\omega}}'$ for $C$ and $D$ is generated by (a $w^{*}$-continuous completely isometric copy of) the linking algebra $\mathcal{L^{\omega}}$ for $A$ and $B$.
  Like in the norm-closed case, we also have that the correspondence $C\mapsto \mathcal{F}(C)= Y\otimes^{\sigma h}_{A}C \otimes^{\sigma h}_{A}X$ taking $W^{*}$-algebras generated by $A$ to  $W^{*}$-algebras generated by $B$ is bijective and order preserving (\cite[Theorems 5.3.5 and 5.3.6]{Kashyap2008}).

%In \cite[section 4]{Blecher2004f}, Blecher and Solel defined a maximal $W^{*}$-algebra $W^{*}_{max}(A)$ of a dual operator algebra $A$. This is a $W^{*}$-algebra $W^{*}_{max}(A)$, together with a weak$^{*}$ continuous completely isometric homomorphism $j: A\to W^{*}_{max}(A)$ whose range generates $W^{*}_{max}(A)$ as a $W^{*}$-algebra, and which has the following universal property: given any normal representation $\pi: A\to B(H)$, there exists a normal $^{*}$-representation $W^{*}_{max}(A)\to B(H)$ extending $\pi$. That is, $W^{*}_{max}(A)$ is a maximal $W^{*}$-cover of $A$.

Let $W^{*}_{max}(A)$ denote the maximal $W^{*}$-algebra of a dual operator algebra $A$, as defined in \cite[section 4]{Blecher2004f}. Using $Z = 
\left( 
\begin{matrix} 
F & F\overbar{\otimes}_{B} \widetilde{X} \\
E \overbar{\otimes}_{A} X & E
\end{matrix}
\right)$, as in the proof of \cref{Moritahardy}, we have the following:
\begin{corollary}
If two $W^{*}$-correspondences $(E,A)$ and $(F,B)$ are weakly Morita equivalent, and $C$ is a $W^{*}$-algebra generated by $H^{\infty}(E)$, then
\begin{itemize}
\item[(1)] $pH^{\infty}(Z)q\overbar{\otimes}_{H^{\infty}(E)}C\overbar{\otimes}_{H^{\infty}(E)}qH^{\infty}(Z)p$  is a $W^{*}$-algebra generated by $H^{\infty}(F)$, which is Morita equivalent to $C$ via the equivalence bimodule $pH^{\infty}(Z)q\overbar{\otimes}_{H^{\infty}(E)}C$. In particular, $pH^{\infty}(Z)q\overbar{\otimes}_{H^{\infty}(E)}W^{*}_{max}(H^{\infty}(E)) \overbar{\otimes}_{H^{\infty}(E)}qH^{\infty}(Z)p=W^{*}_{max}(H^{\infty}(F)) $ is Morita equivalent to $W^{*}_{max}(H^{\infty}(E))$  via the equivalence bimodule $W^{*}_{max}(H^{\infty}(E))$ $\overbar{\otimes}_{H^{\infty}(E)}qH^{\infty}(Z)p$.%\cref{Moritahardy}, \cite[theorem 5.3.5]{Kashyap2008}
\item[(2)]$W^{*}_{max}(H^{\infty}(E))$ and $W^{*}_{max}(H^{\infty}(F))$ are Morita equivalent $W^{*}$-algebras.
\item[(3)] $pH^{\infty}(Z)q\overbar{\otimes}_{H^{\infty}(E)}C\cong \mathcal{F}(C)\overbar{\otimes}_{H^{\infty}(E)}pH^{\infty}(Z)q$. In particular, $pH^{\infty}(Z)q\overbar{\otimes}_{H^{\infty}(E)}$ $W^{*}_{max}(H^{\infty}(E))\cong W^{*}_{max}(H^{\infty}(F)) \overbar{\otimes}_{H^{\infty}(E)}pH^{\infty}(Z)q$.
\end{itemize}
\end{corollary}
\begin{proof}
(1) folows directly from \Cref{Moritahardy} and \cite[Theorem 5.3.5]{Kashyap2008}. Parts (2) and (3) follow from \cref{Moritahardy} and \cite[Theorem 5.2]{Blecher2008a}.
\end{proof}

\section{$W^{*}$-Graph Correspondences} \label{graphChap}
A \emph{directed graph} $G=(G^{0}, G^{1},r,s)$ consists of two countable sets $G^{0}, G^{1}$ and functions $r,s:G^{0}\rightarrow G^{1}$ identifying the range and source of each edge. 
%Let $G=(G^{0},G^{1}, r, s)$ be a directed graph (so $G^{0}$ and $G^{1}$ are countable sets). %The $C^{*}$-correspondence $(E,A)$ associated to $G$ is given by letting $A$ be the $C^{*}$-algebra $c_{0}(G^{0})$ of functions $a: G^{0}\to \mathbb{C}$ which vanish at infinity. $E$ is the norm completion of the space $c_{c}(G^{1})$ of functions of finite support with $(x\cdot a)(e)=x(e)a(s(e))$, $(a\cdot x)(e)=a(r(e))x(e)$ and $\langle x,y\rangle_{A}(v)=\sum\limits_{s(e)=v}\overline{x(e)}y(e)$.
The $W^{*}$-correspondence $(E,A)$ associated to $G=(G^{0},G^{1}, r, s)$ is given by:
\begin{center}
 $A= \ell^{\infty}_{|G^{0}|}  \qquad\qquad E=\{x:G^{1}\to \mathbb{C}\thinspace |\thinspace \sup\limits_{v\in G^{0}}\{\sum\limits_{s(e)=v}|x(e)|^{2}\}<\infty\} $
\end{center}
%By Raeburn graph book pg71 last paragraph, we can just say that E is the w*-completion of the space of functions of finite support on the edges. Then "perhaps" by a w* variation of lemma 2.16 of Raeburn Williams Morita book, E is a Hilbert W* module over A.
The left and right actions are given by $(a\cdot x\cdot b)(e)=a(r(e))x(e)b(s(e))$, where $a,b\in A$ and $x\in E$. The inner product is given by $\langle x,y\rangle_{A}(v)=\sum\limits_{s(e)=v}\overline{x(e)}y(e)$. This definition of a $W^{*}$-graph correspondence is equivalent to the one given by Solel in \cite[pg 3]{Solel2004a}, where the $W^{*}$-correspondence is defined in terms of matrices. Let $I,J$ be indexing sets with $|I|=|G^{0}|$ and $|J|=|G^{1}|$. We will write the elements of $A$ as $a=(a_{i})$ or as $a=\sum\limits_{i\in I}a_{i}\delta _{v_{i}}$, and the elements of $E$ as $(z_{j})$ or as $x=\sum\limits_{j\in J}z_{i}\delta _{e_{j}}$, ($\delta _{v_{i}}$ and $\delta _{e_{j}}$ denote the point masses of a vertex and an edge respectively). Note that $A$ is the $w^{*}$-closure of $ c_{0}(G^{0})$ (by $\ell^{\infty}= c_{0}^{**}$ and Goldstine's theorem). The norm of $A$ is given by $||a||= \sup\limits_{i\in I}|a_{i}|$. The norm of $E$ is given by $||x||=||\langle x, x \rangle_{A}||^{\frac{1}{2}}=(\sup\limits_{v\in G^{0}}\{\sum\limits_{s(e)=v}|x(e)|^{2}\})^{\frac{1}{2}}$. Note that $E$ is a subspace of $\ell^{\infty}_{|G^{1}|}$, which may also be viewed as a disjoint union $\bigsqcup\limits_{v\in G^{0}} \ell^{2}_{|s^{-1}(v)|}$.

For each vertex $v_{i}\in G^{0}$, we have $\delta _{v_{i}}^{2}=\delta _{v_{i}}=\delta _{v_{i}}^{*}$. That is, for each $i\in I$, $\delta _{v_{i}}$ is a projection. If $\sigma$ is a faithful normal representation of $A$ on a Hilbert space $H$, then $\sigma(\delta _{v_{i}})$ is also a projection. So for any $a=(a_{i})\in A$, we have $\sigma (a)=\bigoplus \limits_{i\in I}a_{i}I_{H_{i}}$, a direct sum of uniformly bounded operators on $H=\bigoplus \limits_{i\in I}H_{i}$, the Hilbert space direct sum. Then $\sigma (A)= \bigoplus^{\infty}_{i}\mathbb{C}_{i}I_{H_{i}}$, where $\bigoplus^{\infty}_{i}$ denotes the $\infty$-direct sum (for more information on this sum, see for example \cite[1.2.17]{Blecher2004d}). Since $\sigma$ is faithful, it is isometric ($||\sigma (a)||= \sup\limits_{i}a_{i}I_{H_{i}}=\sup\limits_{i}a_{i}=||a||)$, and so are its amplifications $\sigma _{n}$. So $\sigma$ is completely isometric. The dimension of each Hilbert space $H_{i}$ is the multiplicity $m_{i}$ of the one-dimensional representation $\sigma (\sum \limits^{|G^{0}|}_{i=1}a_{i}\delta _{v_{i}})=a_{i}$. So $\sigma$ is completely determined up to unitarily equivalence by the sequence $(m_{1}, m_{2}, \cdots )$ of these multiplicities. Since $\sigma$ is faithful, $0<m_{i}\leq \infty$. Thus $H$ can be written as $H=\bigoplus \limits_{i\in I}\mathbb{C} ^{m_{i}}$, where $\mathbb{C}^{\infty}$ is interpreted as $\ell^{2}$.

As stated in section \ref{preliminaries}, 
attached to each faithful normal representation $\sigma$ of $A$, there is a dual correspondence $E^{\sigma}$, which is a  $W^{*}$- correspondence over $\sigma(A)'$. Furthermore, 
the ultraweakly continuous representations of $H^{\infty}(E)$ are parametrized by the elements of $\mathbb{D}((E^{\sigma})^{*})$. In the following theorem, we identify the elements of $(E^{\sigma})^{*}$ when $(E,A)$ is a  $W^{*}$-graph correspondence.
\begin{theorem} \label{graphintertwiners}
If $(E,A)$ is a $W^{*}$-graph correspondence and $\sigma:A\to B(H)$ is a faithful normal representation of $A$, then  the elements of $(E^{\sigma})^{*}$ are block matrices $(T_{ij})$ where $T_{ij}\in B(H_{s(e_{j})},H_{v_{i}})$ and  $T_{ij}=0$ if $r(e_{j})\neq v_{i}$.
\end{theorem}
\begin{proof}
Let $x\in E$, $a\in A$ and $h\in H$. $\sigma (a)= \sigma (\sum \limits^{|G^{0}|}_{i=1}a_{i}\delta _{v_{i}}) = \sum \limits^{|G^{0}|}_{i=1}a_{i}\sigma(\delta _{v_{i}})$. So $\sigma(A)=\bigoplus^{\infty}_{i}a_{i}I_{H_{i}}$ and $H=\bigoplus \limits^{|G^{0}|}_{i=1}H_{i}$. Sometimes, for clarity, we might also write $\sigma(A)=\bigoplus^{\infty}_{i}a_{v_{i}}I_{H_{v_{i}}}$ and $H=\bigoplus \limits^{|G^{0}|}_{i=1}H_{v_{i}}$. Let $x\otimes h \in E\otimes_{\sigma} H$.

$x \otimes h = (\sum \limits^{|G^{1}|}_{j=1}z_{j}\delta _{e_{j}}) \otimes h = \sum \limits^{|G^{1}|}_{j=1}z_{j}\delta _{e_{j}}\cdot \delta _{s(e_{j})} \otimes h
= \sum \limits^{|G^{1}|}_{j=1}z_{j}\delta _{e_{j}} \otimes \sigma(\delta _{s(e_{j})}) h 
= \sum \limits^{|G^{1}|}_{i=1}z_{j}\delta _{e_{j}} \otimes h_{s(e_{j})}$.

So $E \otimes _{\sigma}H=\bigoplus \limits^{|G^{1}|}_{j=1} \mathbb{C}\delta _{e_{j}}\otimes H_{s(e_{j})}\cong \bigoplus \limits^{|G^{1}|}_{j=1} \mathbb{C}\otimes H_{s(e_{j})}\cong \bigoplus \limits^{|G^{1}|}_{j=1}H_{s(e_{j})}$. Then $(E^{\sigma})^{*}= \mathcal{I}(\sigma^{E} \circ \varphi, \sigma)\subset B(E \otimes _{\sigma}H,H)=B(\bigoplus \limits^{|G^{1}|}_{j=1}H_{s(e_{j})},\bigoplus \limits^{|G^{0}|}_{i=1}H_{v_{i}})$. 
Let $\eta ^{*} \in (E^{\sigma})^{*}$. So $\eta ^{*} $ is a block matrix $(T_{ij})$, where $T_{ij}\in B(H_{s(e_{j})},H_{v_{i}})$. 
Sometimes, for clarity, we will write $\eta ^{*} =(T_{ij})$ as $\eta ^{*} =(T_{v_{i}e_{j}})$, where $T_{v_{i}e_{j}}\in B(H_{s(e_{j})},H_{v_{i}})$.
\begin{align*}
\sigma^{E} \circ \varphi (a)(x\otimes h)&=a\cdot x\otimes h=( \sum \limits^{|G^{0}|}_{i=1}a_{i}\delta _{v_{i}})\cdot (\sum \limits^{|G^{1}|}_{j=1}z_{j}\delta _{e_{j}})\otimes h=\sum \limits^{|G^{1}|}_{j=1}z_{j}a_{r(e_{j})}\delta _{e_{j}}\otimes h\\
&=\sum \limits^{|G^{1}|}_{j=1}z_{j}a_{r(e_{j})}\delta _{e_{j}}\cdot \delta _{s(e_{j})} \otimes h= \sum \limits^{|G^{1}|}_{j=1}z_{j}a_{r(e_{j})}\delta _{e_{j}} \otimes \sigma(\delta _{s(e_{j})})h\\
&= \sum \limits^{|G^{1}|}_{j=1}z_{j}a_{r(e_{j})}\delta _{e_{j}} \otimes h_{s(e_{j})}
\end{align*}
That is, $\sigma^{E} \circ \varphi (a)(\sum \limits^{|G^{1}|}_{j=1}z_{j}\delta _{e_{j}} \otimes h_{s(e_{j})})=\sum \limits^{|G^{1}|}_{j=1}z_{j}a_{r(e_{j})}\delta _{e_{j}} \otimes h_{s(e_{j})}$ ,or isomorphically, $\sigma^{E} \circ \varphi (a)(\sum \limits^{|G^{1}|}_{j=1}z_{j} h_{s(e_{j})})=\sum \limits^{|G^{1}|}_{j=1}z_{j}a_{r(e_{j})} h_{s(e_{j})}$. So $\sigma^{E} \circ \varphi (a)=\bigoplus \limits^{|G^{1}|}_{j=1}a_{r(e_{j})}I_{s(e_{j})}$. Since an intertwiner $\eta ^{*}\in (E^{\sigma})^{*}$ satisfies $\eta ^{*}(\sigma^{E} \circ \varphi(a))=\sigma(a)\eta ^{*} $, we have $(T_{ij})(\bigoplus \limits^{|G^{1}|}_{j=1}a_{r(e_{j})}I_{s(e_{j})})=(\bigoplus \limits^{|G^{0}|}_{i=1}a_{v_{i}}I_{H_{v_{i}}})(T_{ij})$. So $(a_{r(e_{j})}T_{ij})=(a_{v_{i}}T_{ij})$.

Then, since each edge $e_{j}$ has a unique range $r(e_{j})$, we must have $T_{ij}=0$ if $r(e_{j})\neq v_{i}$. In other words, if we write the blocks $T_{ij}$ as $T_{v_{i}e_{j}}$, we have  that the only (possible) nonzero blocks of an intertwiner $(T_{v_{i}e_{j}})\in (E^{\sigma})^{*}$ are the blocks $T_{r(e_{j})e_{j}} \in B(H_{s(e_{j})},H_{r(e_{j})})$. That is, the only nonzero block of each column $e_{j}$ of $\eta ^{*}\in (E^{\sigma})^{*}$ is the block on row $r(e_{j})$.
\end{proof}
Furthermore, since $||\eta^{*}||=||\eta^{*}\eta||^{\frac{1}{2}}$,  we have $||\eta^{*}||=||(T_{ij})||=||(T_{ij})(T_{ij})^{*}||^{\frac{1}{2}}=||\bigoplus \limits^{|G^{0}|}_{i=1}\sum \limits_{r(e_{j})=v_{i}}T_{ij}T_{ij}^{*}|| $. So $\mathbb{D}((E^{\sigma})^{*})=\{(T_{ij})\enspace | \enspace  T_{ij}\in B(H_{s(e_{j})},H_{v_{i}}), T_{ij}=0\enspace \text{if}\enspace r(e_{j})\neq v_{i}, \text{and}\enspace ||\bigoplus \limits^{|G^{0}|}_{i=1}\sum \limits_{r(e_{j})=v_{i}}T_{ij}T_{ij}^{*}||<1 \}$.
\begin{corollary}\label{centergraph}
If $(E,A)$ is a $W^{*}$-graph correspondence and $\sigma:A\to B(H)$ is a faithful normal representation of $A$, then the elements of $\mathfrak{Z}((E^{\sigma})^{*})$ are block matrices $(T_{ij})$ such that $ T_{ij} =\left\{ \begin{array}{lll}
z_{j}I_{H_{s(e_{j})}} & \quad & \mbox{if }e_{j} \mbox{ is a loop}\\
0 & \quad & \mbox{otherwise}.
\end{array}\right.$
\end{corollary}
\begin{proof}
It follows from \cite[Lemma 4.12 (2)]{Muhly2008b} and the definition of left and right actions in a  $W^{*}$-graph correspondence.
\end{proof}
%So for example, for the graph in the example above, we have
%
%$\mathfrak{Z}((E^{\sigma})^{*})= \{
%\left( 
%\begin{matrix} 
%z_{1}I_{H_{s(e_{1})}} & 0 & 0 & 0 \\
%0 & 0 &  z_{3}I_{H_{s(e_{3})}}  & 0
%\end{matrix} 
%\right)\in 
%\left( 
%\begin{matrix} 
%B(H_{1},H_{1}) & 0 & 0 & 0 \\
%0 & 0 &  B(H_{2},H_{2})  & 0
%\end{matrix} 
%\right)
%\}$ 

\subsection{Commutants of Induced Representations of The Hardy Algebra}
Our next goal is to show that if $(E,A)$ is a $W^{*}$-correspondence derived from a directed graph $G$ and $\sigma, \tau$ are two faithful normal representations of $A$, then $(\sigma^{\ms{F}(E)}(H^{\infty}(E)))'$ is weakly Morita equivalent to $(\tau^{\ms{F}(E)}(H^{\infty}(E)))'$.
\begin{lemma} \label{algebrasme}
The $W^{*}$-algebras $\sigma (A)'$ and $\tau (A)'$ are weakly Morita equivalent.
\end{lemma}
%commutants of von neumann algebras are von neumann algebras, see  von neumann lecture bookmark (home)
\begin{proof}
If $\sigma$ and $\tau$ are faithful normal representations of $A$ on $H$ and $K$ respectively, then
\begin{center}
$H=\bigoplus\limits_{q\in I}H_{q}  \qquad \text{and} \qquad K=\bigoplus\limits_{q\in I}K_{q} \quad$ ,

\medskip
$\sigma (A)=\bigoplus^{\infty}_{q}\mathbb{C} _{q}I_{H_{q}} \qquad \text{and} \qquad  \tau (A)=\bigoplus^{\infty}_{q}\mathbb{C} _{q}I_{K_{q}} \quad$ ,

\medskip
$\sigma (A)'= \bigoplus^{\infty}_{q}B(H_{q}) \qquad \text{and} \qquad  \tau (A)'= \bigoplus^{\infty}_{q}B(K_{q})$.
\end{center}
Let $X=\bigoplus^{\infty}_{q}B(K_{q},H_{q})$. Let $\bigoplus\limits_{q\in I}T_{q},\bigoplus\limits_{q\in I}S_{q} \in X$. We show that $X$ with inner products
\begin{eqnarray*}
 \langle \bigoplus\limits_{q\in I}T_{q},\bigoplus\limits_{q\in I}S_{q} \rangle _{\tau (A)'} &=& \bigoplus\limits_{q\in I}T^{*}_{q}S_{q} \\
 _{\sigma (A)'} \langle \bigoplus\limits_{q\in I}T_{q},\bigoplus\limits_{q\in I}S_{q} \rangle &=& \bigoplus\limits_{q\in I}T_{q}S_{q}^{*} 
\end{eqnarray*}
is a $\sigma (A)'$ -$\tau (A)'$ equivalence bimodule, where the left and right actions are given by regular matrix multiplication. First, we check that the two previous equations do define inner products on $X$. Let $x=\bigoplus\limits_{q\in I}T_{q}, y=\bigoplus\limits_{q\in I}S_{q}, z=\bigoplus\limits_{q\in I}U_{q} \in X$, $a=\bigoplus\limits_{q\in I}R_{q} \in \tau (A)'$ and $ \lambda, \mu \in \mathbb{C}$. Then
\begin{align*}% left aligned
\langle x, \lambda y+\mu z \rangle _{\tau (A)'} &= \langle\bigoplus\limits_{q\in I}T_{q}, \lambda \bigoplus\limits_{q\in I}S_{q}+\mu  \bigoplus \limits_{q\in I} U_{q}\rangle _{\tau (A)'}=\langle\bigoplus\limits_{q\in I}T_{q}, \bigoplus\limits_{q\in I} (\lambda S_{q}+\mu U_{q})\rangle _{\tau (A)'}\\
&=\bigoplus\limits_{q\in I}T_{q}^{*} (\lambda S_{q}+\mu U_{q})=\lambda \bigoplus\limits_{q\in I}T_{q}^{*} S_{q}+\mu \bigoplus\limits_{q\in I}T_{q}^{*} U_{q}\\
&=\lambda \langle \bigoplus\limits_{q\in I}T_{q}, \bigoplus\limits_{q\in I} S_{q}\rangle _{\tau (A)'}+ \mu \langle \bigoplus\limits_{q\in I}T_{q}, \bigoplus\limits_{q\in I} U_{q}\rangle _{\tau (A)'}\\
&= \lambda \langle x, y\rangle _{\tau (A)'} +\mu  \langle x,z \rangle _{\tau (A)'}&
\end{align*}
\begin{align*}% left aligned
\langle x, y \cdot a\rangle _{\tau (A)'} &= \langle\bigoplus\limits_{q\in I}T_{q},  \bigoplus\limits_{q\in I}S_{q}\cdot (\bigoplus\limits_{q\in I}R_{q}) \rangle _{\tau (A)'}=\langle\bigoplus\limits_{q\in I}T_{q},  \bigoplus\limits_{q\in I}S_{q}R_{q}) \rangle _{\tau (A)'}\\
&=\bigoplus\limits_{q\in I}T_{q}^{*}S_{q}R_{q}=(\bigoplus\limits_{q\in I}T_{q}^{*}S_{q}) (\bigoplus\limits_{q\in I}R_{q})&\\
&=\langle\bigoplus\limits_{q\in I}T_{q}, \bigoplus\limits_{q\in I}S_{q}\rangle _{\tau (A)'}\bigoplus\limits_{q\in I}R_{q}=\langle x, y\rangle _{\tau (A)'}a
\end{align*}
\begin{align*}% left aligned
\langle x, y\rangle _{\tau (A)'}^{*} = \langle\bigoplus\limits_{q\in I}T_{q}, \bigoplus\limits_{q\in I}S_{q}\rangle _{\tau (A)'}^{*}= (\bigoplus\limits_{q\in I}T_{q}^{*}S_{q})^{*}=\bigoplus\limits_{q\in I}S_{q}^{*}T_{q}&=\langle\bigoplus\limits_{q\in I}S_{q},  \bigoplus\limits_{q\in I}T_{q}\rangle _{\tau (A)'}\\
&=\langle y, x\rangle _{\tau (A)'}&
\end{align*}
\begin{align*}% left aligned
\langle x, x\rangle _{\tau (A)'} &= \langle\bigoplus\limits_{q\in I}T_{q}, \bigoplus\limits_{q\in I}T_{q}\rangle _{\tau (A)'}= \bigoplus\limits_{q\in I}T_{q}^{*}T_{q}=(\bigoplus\limits_{q\in I}T_{q})^{*}(\bigoplus\limits_{q\in I}T_{q}) \geq 0&
\end{align*} % since in a C*algebra the positive elements are those of the form a*a.
Since each diagonal entry $(T_{q}^{*}T_{q})_{jj}$ of $T_{q}^{*}T_{q}$ is the sum  $\sum\limits_{i}( (T_{q})_{ij})^{2}$ of the squares of the entries on column $j$ of $T_{q}$, we have that $T_{q}^{*}T_{q}=0$ implies $T_{q}=0$. So $\langle x, x\rangle _{\tau (A)'}=\langle\bigoplus\limits_{q\in I}T_{q}, \bigoplus\limits_{q\in I}T_{q}\rangle _{\tau (A)'}=\bigoplus\limits_{q\in I}T_{q}^{*}T_{q}=0$ implies $T_{q}=0$ for all $q \in I$. So $x=\bigoplus\limits_{q\in I}T_{q}=0$. Thus $\langle \cdot , \cdot \rangle _{\tau (A)'}$ is a right inner product on $X$. Similarly, $_{\sigma (A)'} \langle \cdot , \cdot \rangle$ is a left inner product on $X$.

Now we show that $X$ is a  $w^{*}$-full left Hilbert $\sigma (A)'$-module and a  $w^{*}$-full right Hilbert $\tau (A)'$-module. 
Let $M \in \tau (A)'$. So $M=\bigoplus\limits_{q\in I}M_{q}$\quad (where $M_{q} \in B(K_{q})$). Assume for the moment, each Hilbert space $K_{q}$ has finite dimension $n_{q}$. For each $q\in I$, $M_{q}=\sum\limits_{i,j=1}^{n_{q}}m_{ij}^{(q)}E_{ij}^{(q)}$ %is this only true for finite dimensional representations sigma and tau? If M_q is an infinite matrix then there is no integer n_q and I should write M_q is the w* closure of this finite sums ?
, where $m_{ij}^{(q)} \in  \mathbb{C}$ and $\{E_{ij}^{(q)}\}_{i,j}$ is the usual (matrix unit) basis for $B(K_{q})$. That is, $E_{ij}^{(q)}$ is the $n_{q} \times n_{q}$ matrix with $1$ in the $i,j$ entry and zeros everywhere else. Let $\{T_{ij}^{(q)}\}_{i,j}$ be a matrix basis for $B(K_{q},H_{q})$. Then $E_{ij}^{(q)}=T_{1i}^{(q)}\text{}^{*} T_{1j}^{(q)}=\langle T_{1i}^{(q)}, T_{1j}^{(q)} \rangle _{\tau (A)'}$. So  $M_{q}=\sum\limits_{i,j=1}^{n_{q}}m_{ij}^{(q)}E_{ij}^{(q)}=\sum\limits_{i,j=1}^{n_{q}}m_{ij}^{(q)}\langle T_{1i}^{(q)}, T_{1j}^{(q)} \rangle _{\tau (A)'}$.

Thus $M=\bigoplus\limits_{q\in I}M_{q}=\bigoplus\limits_{q\in I}\sum\limits_{i,j=1}^{n_{q}}m_{ij}^{(q)}E_{ij}^{(q)}=\bigoplus\limits_{q\in I}\sum\limits_{i,j=1}^{n_{q}}m_{ij}^{(q)}\langle T_{1i}^{(q)}, T_{1j}^{(q)} \rangle _{\tau (A)'}$.

Likewise, since $E_{ij}^{(q)}=T_{i1}^{(q)}T_{1j}^{(q)}=T_{i1}^{(q)}T_{j1}^{(q)}\text{}^{*}=_{\sigma (A)'} \langle T_{i1}^{(q)}, T_{j1}^{(q)} \rangle$, we have $M=\bigoplus\limits_{q\in I}M_{q}=\bigoplus\limits_{q\in I}\sum\limits_{i,j=1}^{n_{q}}m_{ij}^{(q)}E_{ij}^{(q)}=\bigoplus\limits_{q\in I}\sum\limits_{i,j=1}^{n_{q}}m_{ij}^{(q)}\text{} _{\sigma (A)'}\langle T_{i1}^{(q)}, T_{j1}^{(q)} \rangle$. Then, since for each $q$,  $\mathbb{K}(K_{q})$ is $w^{*}$ dense in $B(K_{q})$ (by Goldstine's theorem), 
% we have $\tau(A)'=\overline{K(_{\sigma(A)'}X)}^{w^{*}}$ and  $\sigma(A)'=\overline{K(X_{\tau(A)'})}^{w^{*}}$  
%(by \cite[8.5.3, 8.1.15, 8.5.13]{Blecher2008a})
%here I am using Merdy 8.5.3 with the isomorphisms at the beginning of the proof of MoritaHardy theorem. By 8.5.13 K(X)=I and tau(A)'=M(K(X))=M(I)= I^{w}.
 we have that $X=\bigoplus\limits_{q\in I}B(K_{q},H_{q})$ is a  $w^{*}$-full left Hilbert $\sigma (A)'$-module and a  $w^{*}$-full right Hilbert $\tau (A)'$-module.
%here we are assuming the representations sigma and tau are finite dimensional. And if one of the M_q is infinite then M_q is the w* closure of finite sums of this form. So then M is the w* closure of the double sums. In other words, the double sums that I wrote are w* dense in the commutants. So the modules are W* full.
% I wrote a proof of $B(H) = \overline{K(H)}^{w^{*}}$ for any Hilbert space, on paper. A second proof is viewing H as a W* Hilbert module over C, B(H) = adjointable operators of(H)=M(K(H))=M(I) =I^{w}. Not sure about this second proof because the K in the second proof refers to thegeneralized compact operators. I don't know if every compact operator is a generalized compact operator.

Next, we show that $\sigma (A)'$ acts as adjointable operators on $X_{\tau (A)'}$ and $\tau (A)'$ acts as adjointable operators on $_{\sigma (A)'}X$. Let $x, y \in X$, $a\in \sigma (A)'$ and $b\in \tau (A)'$. Then
{%
  \abovedisplayskip=0pt \belowdisplayskip=0pt
\begin{align*}
\langle a \cdot x, y \rangle _{\tau (A)'} &= \langle(\bigoplus\limits_{q\in I}N_{q})\cdot \bigoplus\limits_{q\in I}T_{q},\bigoplus\limits_{q\in I}S_{q} \rangle_{\tau (A)'}=\langle\bigoplus\limits_{q\in I}N_{q} T_{q},\bigoplus\limits_{q\in I}S_{q} \rangle_{\tau (A)'}\\
&=\bigoplus\limits_{q\in I}T_{q}^{*} N_{q}^{*}S_{q}=\langle\bigoplus\limits_{q\in I} T_{q},\bigoplus\limits_{q\in I}N_{q}^{*}S_{q} \rangle_{\tau (A)'}\\
&= \langle\bigoplus\limits_{q\in I}T_{q}, (\bigoplus\limits_{q\in I}N_{q})^{*} \cdot \bigoplus\limits_{q\in I}S_{q} \rangle_{\tau (A)'}=\langle x, a^{*} \cdot y\rangle _{\tau (A)'} 
\end{align*}
}%
and
{%
  \abovedisplayskip=0pt \belowdisplayskip=0pt
\begin{align*}
_{\sigma (A)'}\langle x \cdot b, y \rangle &=_{\sigma (A)'} \langle(\bigoplus\limits_{q\in I}T_{q})\cdot (\bigoplus\limits_{q\in I}M_{q}),\bigoplus\limits_{q\in I}S_{q} \rangle=_{\sigma (A)'} \langle\bigoplus\limits_{q\in I}T_{q} M_{q},\bigoplus\limits_{q\in I}S_{q} \rangle\\
&=\bigoplus\limits_{q\in I}T_{q} M_{q}S_{q}^{*}=_{\sigma (A)'} \langle\bigoplus\limits_{q\in I} T_{q},\bigoplus\limits_{q\in I}S_{q}M_{q}^{*} \rangle\\
&= _{\sigma (A)'}\langle\bigoplus\limits_{q\in I}T_{q}, (\bigoplus\limits_{q\in I}S_{q}) \cdot (\bigoplus\limits_{q\in I}M_{q})^{*} \rangle=_{\sigma (A)'} \langle x,y \cdot b^{*} \rangle
\end{align*}
}%

Next, we show that the two inner products are compatible.

For all $x=\bigoplus\limits_{q\in I}T_{q}, y=\bigoplus\limits_{q\in I}S_{q}, z=\bigoplus\limits_{q\in I}U_{q} \in X$, we have

\begingroup
\allowdisplaybreaks
{%
  \abovedisplayskip=0pt \belowdisplayskip=0pt
\begin{align*}
_{\sigma (A)'}\langle x, y \rangle \cdot z &= _{\sigma (A)'} \langle\bigoplus\limits_{q\in I}T_{q},\bigoplus\limits_{q\in I}S_{q} \rangle \cdot\bigoplus\limits_{q\in I}U_{q}\\
&= \bigoplus\limits_{q\in I}T_{q} S_{q}^{*} \cdot (\bigoplus\limits_{q\in I}U_{q})\\
&=\bigoplus\limits_{q\in I}T_{q} S_{q}^{*}U_{q}\\
&=(\bigoplus\limits_{q\in I} T_{q})\cdot \bigoplus\limits_{q\in I}S_{q}^{*}U_{q}\\
&= \bigoplus\limits_{q\in I}T_{q}\cdot \langle \bigoplus\limits_{q\in I}S_{q}, \bigoplus\limits_{q\in I}U_{q} \rangle _{\tau (A)'}\\
&=x \cdot  \langle y,z\rangle _{\tau (A)'}
\end{align*}
}%
\endgroup
Finally, each $B(K_{q},H_{q})$ is the dual of the space $S^{1}((H_{q},K_{q})$ of trace class operators. %(regarded as a subspace of $B(K_{q},H_{q})^{*}$) from Merdy 1.4.5
So $X=\bigoplus^{\infty}_{q}B(K_{q},H_{q}) =(\bigoplus^{1}_{q\in I}S^{1}(K_{q},H_{q}))^{*}$ % this equality is given in Merdy 1.4.13
is a dual space. Here, $\bigoplus^{1}_{q\in I}$ denotes the $1$-direct sum. For more information on this direct sum and its $\ell^{1}$-norm, see \cite[1.4.13]{Blecher2004d} for example.
 Thus $X$ is a 
 $\sigma (A)'$-$\tau (A)'$ $W^{*}$-equivalence bimodule. That is, $\sigma (A)'$ and $\tau (A)'$ are weakly Morita equivalent.  
%Here I am using Merdy 8.5.7
\end{proof}
\begin{theorem}\label{moritadual}
If ($E,A$) is a $W^{*}$-graph correspondence and $\sigma :A \to B(H)$, $\tau :A \to B(K)$ are faithful normal representations of $A$, then $(E^{\sigma}$, $\sigma(A)')\stackrel{WME}{\sim}_{X} (E^{\tau}$, $\tau(A)')$.
\end{theorem}
\begin{proof}
Let $X=\bigoplus \limits ^{|G^{0}|}_{j=1}B(K_{i},H_{i})$ be the $\sigma(A)'$-$\tau(A)'$ equivalence bimodule given in \cref{algebrasme}. We show that the map
\begingroup
\allowdisplaybreaks
{%
  \abovedisplayskip=0pt \belowdisplayskip=1pt
\begin{align*}
  \varphi \colon  _{\tau(A)'} \widetilde{X}\overbar{\otimes}_{\sigma(A)'} E^{\sigma}\overbar{\otimes}_{\sigma(A)'}X_{\tau(A)'} \quad &\longrightarrow \quad _{\tau(A)'}E^{\tau}_{\tau(A)'}\\
  \widetilde{x}\otimes \eta \otimes y \quad &\longmapsto \quad (I_{E} \otimes x^{*})\eta y
\end{align*} 
}%
\endgroup

gives a $W^{*}$-correspondence isomorphism% where the map between the algebras of the correspondences is the identity map on \tau(A)'
, where $x, y \in X$ and $(I_{E} \otimes x^{*}):E \overbar{\otimes}_{\sigma(A)}H \to E \overbar{\otimes}_{\tau(A)}K$ is defined by $(I_{E} \otimes x^{*})(\xi \otimes h)=\xi \otimes x^{*}h$. 

Let $x,y \in \thinspace_{\sigma(A)'}X_{\tau(A)'} \thinspace$, $\eta \in E^{\sigma}$, $a,b \in \tau(A)'$. Note that $\varphi$ is well defined, since for any $c,d \in \sigma(A)' $, we have $\varphi(\widetilde{x} \cdot c\otimes \eta \cdot d \otimes y)=\varphi(\widetilde{c^{*}x}\otimes \eta d \otimes y)= (I_{E} \otimes x^{*}c)\eta dy=(I_{E} \otimes x^{*})(I_{E} \otimes c)\eta dy=(I_{E} \otimes x^{*})(c \cdot \eta)(d \cdot y)=\varphi(\widetilde{x} \otimes c \cdot \eta \otimes d\cdot y)$. 
\begingroup
\allowdisplaybreaks
\begin{align*}
 \varphi(a \cdot ( \widetilde{x}\otimes \eta \otimes y)\cdot b) &=  \varphi((a \cdot \widetilde{x})\otimes \eta \otimes (y\cdot b)) \\
 &= \varphi(( \widetilde{x \cdot a^{*}}\otimes \eta \otimes (y\cdot b)) \\
 &= \varphi(( \widetilde{xa^{*}}\otimes \eta \otimes (yb)) \\
 &= (I_{E} \otimes ax^{*})\eta yb\\
 &= (I_{E} \otimes a)(I_{E} \otimes x^{*})\eta yb\\
 &= a\cdot ((I_{E} \otimes x^{*})\eta y)\cdot b\\
  &= a\cdot \varphi(\widetilde{x}\otimes \eta \otimes y)\cdot b
\end{align*}
\endgroup
Let $\widetilde{x_{1}}\otimes \eta_{1} \otimes y_{1}$, $\widetilde{x_{2}}\otimes \eta_{2} \otimes y_{2} \in \widetilde{X} \overbar{\otimes} E^{\sigma} \overbar{\otimes} X$. Then
\begingroup
\allowdisplaybreaks
\begin{align*}
\langle \varphi(\widetilde{x_{1}}\otimes \eta_{1} \otimes y_{1}),\varphi(\widetilde{x_{2}}\otimes \eta_{2} \otimes y_{2})\rangle _{\tau(A)'}
&= \langle (I_{E} \otimes x^{*}_{1}) \eta_{1}y_{1},(I_{E} \otimes x^{*}_{2}) \eta_{2}y_{2}\rangle _{\tau(A)'}\\
&= y^{*}_{1}\eta^{*}_{1}(I_{E} \otimes x_{1})(I_{E} \otimes x^{*}_{2})\eta_{2}y_{2}\\
&= y^{*}_{1}\eta^{*}_{1}(I_{E} \otimes x_{1}x^{*}_{2})\eta_{2}y_{2}\\
&= \langle y_{1},\eta^{*}_{1}(I_{E} \otimes x_{1}x^{*}_{2})\eta_{2}y_{2}\rangle _{\tau(A)'}\\
&= \langle y_{1},(\eta^{*}_{1}(I_{E} \otimes x_{1}x^{*}_{2})\eta_{2})\cdot y_{2}\rangle _{\tau(A)'}\\
&= \langle y_{1},\langle \eta_{1},(I_{E} \otimes x_{1}x^{*}_{2})\eta_{2}\rangle _{\sigma(A)'} \cdot y_{2}\rangle _{\tau(A)'}\\
&= \langle y_{1},\langle \eta_{1},(x_{1}x^{*}_{2})\cdot\eta_{2}\rangle _{\sigma(A)'} \cdot y_{2}\rangle _{\tau(A)'}\\
&= \langle y_{1},\langle \eta_{1},\thinspace  _{\sigma(A)'} \langle x_{1},x_{2}\rangle \cdot\eta_{2}\rangle _{\sigma(A)'} \cdot y_{2}\rangle _{\tau(A)'}\\
&= \langle y_{1},\langle \eta_{1}, \langle \widetilde{x_{1}},\widetilde{ x_{2}} \rangle  _{\sigma(A)'}\cdot\eta_{2}\rangle _{\sigma(A)'} \cdot y_{2}\rangle _{\tau(A)'}\\
&= \langle y_{1},\langle \widetilde{x_{1}} \otimes \eta_{1},\widetilde{x_{2}} \otimes \eta_{2}\rangle _{\sigma(A)'} \cdot y_{2}\rangle _{\tau(A)'}\\
&= \langle \widetilde{x_{1}}\otimes \eta_{1} \otimes y_{1},\widetilde{x_{2}}\otimes \eta_{2} \otimes y_{2}\rangle _{\tau(A)'}
\end{align*}
\endgroup
%need to show that the left inner product is also preserved.

That is, $\varphi$ preserves the inner product. So it is isometric, hence injective with closed range. Now we show $\varphi$ is surjective. Each element $S\in E^{\tau}$ has $|L|$ nonzero blocks $S_{q}$, where $L$ is a set with $|G^{0}|\leq|L|\leq |G^{1}|$. Each $S_{q} \in B(K_{s(e_{q})},K_{r(e_{q})})$. For $Q \in B(K_{s(e_{q})},H_{s(e_{q})}), P\in B(H_{s(e_{q})},H_{r(e_{q})}), R\in B(H_{g},K_{g})$, $RPQ \in B(K_{s(e_{q})},K_{r(e_{q})})$. Let $M_{Q} \in X$ be an element with all zero blocks except for $Q$, $M_{P} \in E^{\sigma}$ be an element with all zero blocks except for $P$, and $M_{R} \in (I_{E} \otimes \til{X})$ be an element with all zero blocks except for $R$. Assume for the moment that all Hilbert spaces  $H_{j}$ and $K_{j}$ are finite dimensional. That is, the multiplicity of the representation of $\delta_{v}$ is finite for all $v\in G^{0}$. $S_{q}=\bigoplus\limits_{i,j}s_{ij}E_{ij}$, where $s_{ij} \in  \mathbb{C}$ and $\{E_{ij}\}_{i,j}$ is a matrix basis for $B(K_{s(e_{q})},K_{r(e_{q})})$. Let $\{T_{ij}\}_{i,j}$ be a matrix basis for $B(K_{s(e_{q})},H_{s(e_{q})})$, $\{Y_{ij}\}_{i,j}$ be a matrix basis for $B(H_{s(e_{q})},H_{r(e_{q})})$, and $\{Z_{ij}\}_{i,j}$ be a matrix basis for $B(H_{r(e_{q})},K_{r(e_{q})})$. Then $E_{ij}=Z_{ij}Y_{j1}T_{1j}$. So for $x, y \in X, \eta \in  E^{\sigma}$and $H, K$ finite dimensional, the products $(I_{E}\otimes \til{x})\eta y$ span $E^{\tau}$. If the representation of $\delta_{s(e_{q})}$ or $\delta_{r(e_{q})}$ on $K$ is not finite dimensional then since
 $B(K_{s(e_{q})},K_{r(e_{q})})=\mathbb{K}(K_{s(e_{q})},K_{r(e_{q})})^{**}$, the span of the finite dimensional products $E_{ij}=Z_{ij}Y_{j1}T_{1j}$ is $w^{*}$-dense in $B(K_{s(e_{q})},K_{r(e_{q})})$ (by Goldstine's theorem).
Summing over all $q\in L$, we have that the span of these finite dimensional products is $w^{*}$-dense in $E^{\tau}$. So $\varphi$ is surjective, %since I had already said the range of phi is closed.
 thus a $W^{*}$-correspondence isomorphism.

Since by \cref{mA}, $X \overbar{\otimes}_{\tau(A)'} \widetilde{X}\cong \sigma(A)'$ and $E^{\sigma}\cong \sigma(A)'\overbar{\otimes}_{\sigma(A)'}E^{\sigma}$ as $W^{*}$-correspondences, we have: 
\begin{align*}
E^{\sigma}\overbar{\otimes}_{\sigma(A)'}X 
\cong \sigma(A)'\overbar{\otimes}_{\sigma(A)'}E^{\sigma} \overbar{\otimes}_{\sigma(A)'}X
\cong X \overbar{\otimes}_{\tau(A)'} \widetilde{X}\overbar{\otimes}_{\sigma(A)'}E^{\sigma}\overbar{\otimes} _{\sigma(A)'}X
\cong X \overbar{\otimes}_{\tau(A)'}E^{\tau}
\end{align*}
as  $W^{*}$-correspondences. So $(E^{\sigma}$, $\sigma(A)')\stackrel{WME}{\sim}_{X} (E^{\tau}$, $\tau(A)')$ . 
\end{proof}
% To generalize this theorem to any W* corresondence (not just graph correspondence), the only part that is missing is showing that the map \varphi is surjective no mater what the equivalence bimodule X is
%\cref{mA} and \Cref{mA}
In \cite{Muhly1999}, Muhly and Solel defined the \emph{induced representations} $\rho$ of $H^{\infty}(E)$, which play a central role in the study of Hardy algebras. Indeed, these induced representations (in the sense of Rieffel \cite{R1974b}) appear in most of the work related to Hardy algebras. In \cite{Muhly2009}, Muhly and Solel showed how the commutant of $\rho(H^{\infty}(E))$ can be expressed in terms of induced representations of $H^{\infty}(E^{\sigma})$. More precisely, let $\sigma:A \to B(H)$ be a normal representation of $A$ on a Hilbert space $H$ and form the Hilbert space $\ms{F}(E) \otimes _{\sigma}H$. The induced covariant representation of $E$ determined by $\sigma$ is the representation $(T, \varphi _{\infty} \otimes I_{H})$ where $T: E \to B(\ms{F}(E) \overbar{\otimes} _{\sigma}H)$ is defined by $T(\xi)(\eta \otimes h)=(\xi \otimes \eta )\otimes h$ % =T_{\xi}(\eta) \otimes  h
 for $\xi \in E$ and $\eta \otimes h \in \ms{F}(E) \overbar{\otimes} _{\sigma}H$. The representation of $H^{\infty}(E)$,  \emph{induced} by $\sigma$, denoted by $\sigma ^{ \ms{F}(E)}$, is the integrated form of $(T, \varphi _{\infty} \otimes I_{H})$. For $X \in H^{\infty}(E)$, $\sigma ^{ \ms{F}(E)}(X)$ is also written as $X \otimes I_{H}$. Define a map $U: \ms{F}(E^{\sigma})\overbar{\otimes} _{\iota} H \to \ms{F}(E)\overbar{\otimes} _{\iota} H$ (where $\iota$ denotes the identity representation of $\sigma(A)'$ in $B(H)$) by
\begin{align*}
U(\eta _{1} \otimes \eta _{2} \otimes \cdots \eta _{n} \otimes h)=(I_{E^{\otimes (n-1)}} \otimes \eta _{1})(I_{E^{\otimes (n-2)}} \otimes \eta _{2}) \cdots (I_{E}\otimes \eta _{n-1}) \eta _{n}h
\end{align*}
By \cite[Lemma 3.8]{Muhly2004a}, $U$ is a Hilbert space isometric isomorphism and by \cite[Theorem 3.9]{Muhly2004a}, the representation $\rho$ of $H^{\infty}(E^{\sigma})$ on $\ms{F}(E) \overbar{\otimes} _{\sigma}H$, defined by the formula
\begin{align*}
\psi (X)=U\iota ^{\ms{F}(E^{\sigma})}(X)U^{*}
\end{align*}
is an ultraweakly homeomorphic, completely isometric isomorphism from $H^{\infty}(E^{\sigma})$ onto $(\sigma ^{\ms{F}(E)}(H^{\infty}(E)))'$. Likewise, the map $\upsilon$, defined  by 
\begin{align*}
\upsilon (X)=U^{*}\sigma ^{\ms{F}(E)}(X)U
\end{align*}
is an ultraweakly homeomorphic, completely isometric isomorphism from $H^{\infty}(E)$ onto $(\iota ^{\ms{F}(E^{\sigma})}(H^{\infty}(E^{\sigma})))'$.
\begin{theorem}\label{moritainduced}
If ($E,A$) is a $W^{*}$-graph correspondence and $\sigma :A \to B(H)$, $\tau :A \to B(K)$ are faithful normal representations of $A$, then $(\sigma ^{\ms{F}(E)}(H^{\infty}(E)))'\stackrel{WME}{\sim}(\tau ^{\ms{F}(E)}(H^{\infty}(E)))'$.
\end{theorem}
\begin{proof}
By \cref{moritadual}, $(E^{\sigma}$, $\sigma(A)')\stackrel{WME}{\sim} (E^{\tau}$, $\tau(A)')$. Then by \cref{Moritahardy}, $ H^{\infty}(E^{\sigma}))\stackrel{WME}{\sim}H^{\infty}(E^{\tau})$. So  by the isomorphism $\psi$ above, we have
\begin{align*}
(\sigma ^{\ms{F}(E)}(H^{\infty}(E)))'\cong H^{\infty}(E^{\sigma}))\stackrel{WME}{\sim}H^{\infty}(E^{\tau})\cong (\tau ^{\ms{F}(E)}(H^{\infty}(E)))'
\end{align*}
\end{proof}
Note also that if ($E,A$) is a graph correspondence and $\sigma :A \to B(H)$, $\tau :A \to B(K)$ are faithful normal representations of $A$, then the map $\upsilon$ above, gives us
\begin{align*}
(\iota ^{\ms{F}(E^{\sigma})}(H^{\infty}(E^{\sigma})))'
\cong (\iota ^{\ms{F}(E^{\tau})}(H^{\infty}(E^{\tau})))'.
\end{align*}
%since (\iota ^{\ms{F}(E^{\sigma})}(H^{\infty}(E^{\sigma})))' \cong  H^{\infty}(E) \cong (\iota ^{\ms{F}(E^{\tau})}(H^{\infty}(E^{\tau})))'
\subsection{Morita Equivalence of $W^{*}$-Graph Correspondences}
Let $X$ be a countable set, $A=C(X)$ (with the sup norm) and let $_{C(X)}\mathfrak{X}_{C(X)}$ = $_{A}\mathfrak{X}_{A}$ be a $W^{*}$-equivalence bimodule. By \cite[Theorem 3.11]{Paschke1973} and Zorn's lemma, $\mathfrak{X}$ has an orthonormal basis $\mathcal{A}$ consisting of mutually orthogonal non zero partial isometries. That is, for each $e_{i}\in \mathcal{A}$, $\langle e_{i}, e_{i}\rangle$ is a nonzero orthogonal projection in $A$, and for each $g\in \mathfrak{X}$, $g=\sum\limits_{i} e_{i}\langle e_{i},g \rangle$. %Merdy 8.5.23
In particular, $\sum\limits_{i} \Uptheta _{e_{i}, e_{i}}=I_{\mathfrak{X}}$ where $\Uptheta _{e_{i}, e_{i}}$ is the usual rank-one operator in $\mathbb{K}(\mathfrak{X})$. The elements of $\mathcal{A}$ are linearly independent, otherwise there would be $e_{j} \in \mathcal{A}$ such that $e_{j}=\sum\limits_{i\neq j}z_{i}e_{i}$ ($z_{i} \in \mathbb{C}$). But then we would have $0< \langle e_{j}, e_{j} \rangle _{A}< \langle e_{j}, \sum\limits_{i\neq j}z_{i}e_{i} \rangle _{A}=\sum\limits_{i\neq j}z_{i}< \langle e_{j}, e_{i} \rangle _{A}= \sum\limits_{i\neq j}z_{i}(0)=0$. %, which is is a contradiction.

Since  $A= \ell^{\infty}$ can be identified with $C(\beta\mathbb{N})$, where $\beta\mathbb{N}$ denotes the Stone Cech compactification of $\mathbb{N}$, the maximal ideals of $C(X)$ are $\{I_{x}\}_{x\in X}$ where 
$I_{x}=\{\sum\limits_{y\in X}a_{y}\delta_{y} :a_{y}\in \mathbb{C}, \text{ }\text{sup}|a_{y}|< \infty \text{ }\text{and}\textbf{ } y\neq x\}$. The maximal $C(X)$-$C(X)$-submodules of $\mathfrak{X}$ are $\{\mathfrak{X}_{j}\}_{j\in \{1, \cdots, n\}}$ where $\mathfrak{X}_{j}=\{\sum\limits_{i}z_{i}e_{i}:z_{i}\in \mathbb{C},\text{ }\text{sup}|z_{i}|< \infty \text{ } \text{and}\text{ }i\neq j\}$. Since the Rieffel correspondence of $_A\mathfrak{X}_A$, pairs maximal ideals of $A=C(X)$ with maximal submodules of $\mathfrak{X}$, we have that dim($\mathfrak{X})=|\mathcal{A}|=$dim$(C(X))=|X|$. %So the orthonormal basis \mathcal{A} of \mathfrak{X} is also a vector basis for \mathfrak{X}.

If the corresponding submodule (under the Rieffel correspondence) for the maximal ideal $I_{x}$ is the maximal submodule $\mathfrak{X}_{j}$, then $\mathfrak{X}_{j}=\mathfrak{X} \cdot I_{x}$ \cite[Lemma 3.23]{Raeburn1998}. So $e_{i}\cdot \delta_{y\neq x} \neq e_{j}$ for all $e_{i} \in \mathcal{A}$. But by Cohen's factorization theorem, $e_{j}=e\cdot a$ for some $e\in \mf{X}$, $a\in C(X)$. So we must have  $e_{i}\cdot \delta _{x}=e_{j}$ for some $e_{i}\in \ms{A}$. %since $e_{i}\cdot \delta_{y\neq x} \neq e_{j}$ for all $e_{i} \in \mathcal{A}$
 Then we have $e_{i}=e_{j}$ (otherwise we would have $e_{i}\cdot \delta_{x}=e_{j}$ for $i\neq j$ and $0<\langle e_{j},e_{j}\rangle_{C(X)}=\langle e_{j},e_{i}\cdot \delta_{x}\rangle_{C(X)}=\langle e_{j},e_{i}\rangle_{C(X)}\cdot \delta_{x}=0\cdot \delta_{x}=0$. % Contradiction).
  Thus the element $x\in X$ (and therefore $\delta_{x}\in C(X)$) gets uniquely paired up with the element $e_{j}\in \ms{A}$. Likewise, each basis element $\delta_{y}\in C(X)$ gets uniquely paired up with a basis element $e_{i} \in \mf{X}$ by the right action relation $e_{i}\cdot \delta_{y}=e_{i}$. So we have a bijection $R$ between the basis elements $\{e_{i} \}$ in $\mf{X}$ and the basis elements $\{\delta_{y}\}$ in $C(X)$. Applying the same analysis to the Rieffel correspondence between $C(X)$ and $\mf{X}$, but now with $C(X)$ giving the left action on $\mf{X}$, we have another bijection $L:\{\delta_{y}\}_{y\in X} \to\ms{A}$. Thus $\sigma=R \circ L$ is a permutation of $\{\delta_{x}:x\in X\}$, % L maps delta_{X} to e_{i} and  R maps e_{i}=L(delta_{x}) to delta_{sigma(x)}
or equivalently, $\sigma$ is  a permutation of $X$ given by the Rieffel correspondence of $_{C(X)}\mathfrak{X}_{C(X)}$. Note that $\delta_{y}\cdot L(\delta_{x})\cdot \delta _{w}=L(\delta _{x})$ if $y=x$ and $w=\sigma(x)$. If we let $_{A}{A_{\sigma}} _{A}$ denote the algebra $A$ (viewed as a bimodule over itself) with a modified right action and right inner product (given by $\sigma$), we obtain the following result:
\begin{lemma}\label{Asigma}
$_{C(X)}\mathfrak{X}_{C(X)}={_A\mathfrak{X}_A}$ is a $W^{*}$-equivalence bimodule if and only if $_A\mathfrak{X}_A$ is of the form $_{A}{A_{\sigma}} _{A}$, where $\sigma=R \circ L$ is the permutation given by the Rieffel correspondence of $_A\mathfrak{X}_A$.
\end{lemma}
\begin{proof}
Let $_{A}{A_{\sigma}} _{A}$ denote the algebra $A$ (viewed as a $W^{*}$-bimodule over itself) with the right and left actions given by:
\begin{align*}
\delta _{x}\cdot \delta _{y}=
\begin{cases} 
\delta _{x} & \text{if } y=\sigma (x) \\
0 & \text{otherwise} 
\end{cases}
\qquad \qquad
\delta _{y}\cdot \delta _{x}=
\begin{cases} 
\delta _{x} & \text{if } y=x\\
0 & \text{otherwise} 
\end{cases}
\end{align*}
and the right and left inner products given by:
\begin{align*}
\langle\delta _{x}, \delta _{y}\rangle _{A}=
\begin{cases} 
\delta _{\sigma (x)} & \text{if } x=y \\
0 & \text{otherwise} 
\end{cases}
\qquad \qquad
_{A}\langle\delta _{x}, \delta _{y}\rangle=
\begin{cases} 
\delta _{x} & \text{if } x=y\\
0 & \text{otherwise} 
\end{cases}
\end{align*}
So that if $\sum\limits_{x\in X}a_{x}\delta_{x}, \sum\limits_{x\in X}c_{x}\delta_{x}\in A_{\sigma}$ and $\sum\limits_{y\in X}b_{y}\delta_{y}\in A$, then the right and left actions and inner products are:
\begin{gather*} 
\sum\limits_{x\in X}a_{x}\delta_{x}\cdot \sum\limits_{y\in X}b_{y}\delta_{y}=\sum\limits_{x\in X}a_{x}b_{\sigma(x)}\delta_{x}\\
\sum\limits_{y\in X}b_{y}\delta_{y}\cdot \sum\limits_{x\in X}a_{x}\delta_{x}=\sum\limits_{x\in X}a_{x}b_{x}\delta_{x}\\
\langle\sum\limits_{x\in X}a_{x}\delta_{x}, \sum\limits_{y\in X}c_{y}\delta_{y}\rangle _{A}=\sum\limits_{x\in X}\overline{a_{x}}c_{x}\delta_{\sigma(x)}=\sum\limits_{x\in X}\overline{a_{\sigma ^{-1} (x)}}c_{\sigma ^{-1}(x)}\delta_{x}\\
_{A}\langle\sum\limits_{x\in X}a_{x}\delta_{x}, \sum\limits_{y\in X}c_{y}\delta_{y}\rangle=\sum\limits_{x\in X}{a_{x}}\overline{c_{x}}\delta_{x}
\end{gather*}
First we check that  $_{A}{A_{\sigma}} _{A}$ is a $W^{*}$-equivalence bimodule:

Since $\{\delta_{x}:x\in X\}$ spans $C(X)=A$ and $\langle\delta_{\sigma ^{-1}(x)}, \delta_{\sigma ^{-1}(x)}\rangle _{C(X)}=\delta_{x}$ and  $_{C(X)}\langle\delta_{x}, \delta_{x}\rangle=\delta_{x}$, we have that $A_{\sigma}$ is a $w^{*}$-full left Hilbert $A$-module and a $w^{*}$-full right Hilbert $A$-module.

Let $s,t \in A_{\sigma}$, $a\in A$.
\begin{align*}
\langle a\cdot s, t\rangle _{A} &= \langle (\sum\limits_{y\in X}a_{y}\delta_{y})\cdot (\sum\limits_{x\in X}z_{x}\delta_{x}), \sum\limits_{w\in X}t_{w}\delta_{w}\rangle _{A}\\
&= \langle \sum\limits_{x\in X}a_{x}z_{x}\delta_{x}, \sum\limits_{w\in X}t_{w}\delta_{w}\rangle _{A}\\
&= \sum\limits_{x\in X}\overline{a_{x}z_{x}}t_{x}\delta_{\sigma(x)}\\
&= \langle \sum\limits_{x\in X}z_{x}\delta_{x}, \sum\limits_{w\in X}\overline{a_{w}}t_{w}\delta_{w}\rangle _{A}\\
&= \langle \sum\limits_{x\in X}z_{x}\delta_{x}, (\sum\limits_{y\in X}\overline{a_{y}}\delta_{y})\cdot (\sum\limits_{w\in X}t_{w}\delta_{w})\rangle _{A}\\
&= \langle s, a^{*}\cdot t\rangle _{A}
\end{align*}
\begingroup
\allowdisplaybreaks
\begin{align*}
_{A}\langle s\cdot a, t\rangle &=_{A}\langle (\sum\limits_{x\in X}z_{x}\delta_{x})\cdot (\sum\limits_{y\in X}a_{y}\delta_{y}), \sum\limits_{w\in X}t_{w}\delta_{w}\rangle\\
&= _{A}\langle \sum\limits_{x\in X}a_{\sigma(x)}z_{x}\delta_{x}, \sum\limits_{w\in X}t_{w}\delta_{w}\rangle \\
&= \sum\limits_{x\in X}a_{\sigma(x)}z_{x}\overline{t_{x}}\delta_{x}\\
&= _{A}\langle \sum\limits_{x\in X}z_{x}\delta_{x}, \sum\limits_{w\in X}\overline{a_{\sigma(w)}}t_{w}\delta_{w}\rangle \\
&= _{A}\langle \sum\limits_{x\in X}z_{x}\delta_{x}, (\sum\limits_{w\in X}t_{w}\delta_{w})\cdot (\sum\limits_{y\in X}\overline{a_{y}}\delta_{y})\rangle \\
&= _{A}\langle s,t\cdot a^{*}\rangle
\end{align*}
\endgroup
Let $r,s,t\in A_{\sigma}$.
\begingroup
\allowdisplaybreaks
\begin{align*}
_{A}\langle r,s\rangle \cdot t &=_{A}\langle \sum\limits_{y\in X}r_{y}\delta_{y}, \sum\limits_{x\in X}z_{x}\delta_{x}\rangle\cdot \sum\limits_{w\in X}t_{w}\delta_{w}\\
&= (\sum\limits_{x\in X}r_{x}\overline{z_{x}}\delta_{x})\cdot \sum\limits_{w\in X}t_{w}\delta_{w} \\
&= \sum\limits_{x\in X}r_{x}\overline{z_{x}}t_{x}\delta_{x}\\
&= \sum\limits_{y\in X}r_{y}\delta_{y}\cdot (\sum\limits_{x\in X}\overline{z_{x}}t_{x}\delta_{\sigma(x)}) \\
&= (\sum\limits_{y\in X}r_{y}\delta_{y})\cdot \langle\sum\limits_{x\in X}z_{x}\delta_{x}, \sum\limits_{w\in X}t_{w}\delta_{w}\rangle _{A} \\
&= r\cdot\langle s,t\rangle _{A}
\end{align*}
\endgroup
Since $A$ has an operator space predual (being a $W^{*}$-algebra), it is a selfdual $C^{*}$-module over itself.
Thus $_{A}{A_{\sigma}} _{A}$ is a $W^{*}$-equivalence bimodule. %A is a dual space since it is a W* algebra. So {A_{\sigma} is a W* equivalence bimodule.

Let $\psi :{_{A}{A_{\sigma}} _{A}}\to {_{A}\mf{X}_{A}}$ be the linear extension of the bijection $L:\{\delta_{y}\}_{y\in X} \to\ms{A}$ that we encountered above when we studied the Rieffel correspondence of $_{A}\mf{X} _{A}$. For any element $e=\sum\limits_{x\in X}z_{x}\delta_{x}\in A_{\sigma}$, we have $\psi(e)=\psi (\sum\limits_{x\in X}z_{x}\delta_{x})=\sum\limits_{x\in X}z_{x}L(\delta_{x})$. We show now that $\psi$ is a $W^{*}$-equivalence bimodule isomorphism. %where the map in the correspondence isomorphism between the algebras of the corresondences, is the identity map.
Recall that $\delta_{y}\cdot L(\delta_{x})\cdot \delta _{w}=L(\delta _{x})$ in $ _{A}\mf{X} _{A}$ if $y=x$ and $w=\sigma(x)$. Let $a,b \in A$ and $e\in A_{\sigma}$. Then
\begingroup
\allowdisplaybreaks
\begin{align*}
\psi (a\cdot e \cdot b)&= \psi(\sum\limits_{y\in X}a_{y}\delta_{y}\cdot \sum\limits_{x\in X}z_{x}\delta_{x}\cdot \sum\limits_{w\in X}b_{w}\delta_{w})\\
&=\psi(\sum\limits_{x\in X}a_{x}z_{x}b_{\sigma(x)}\delta_{x})= \sum\limits_{x\in X}\psi (a_{x}z_{x}b_{\sigma(x)}\delta_{x}) \\
&= \sum\limits_{x\in X}a_{x}z_{x}b_{\sigma(x)}L(\delta_{x})\\
&= \sum\limits_{y\in X}a_{y}\delta_{y}\cdot \sum\limits_{x\in X}z_{x}L(\delta_{x})\cdot \sum\limits_{w\in X}b_{w}\delta_{w} \\
&= (\sum\limits_{y\in X}a_{y}\delta_{y})\cdot \psi (\sum\limits_{x\in X}z_{x}\delta_{x})\cdot \sum\limits_{w\in X}b_{w}\delta_{w} \\
&= a\cdot\psi (e)\cdot b
\end{align*}
\endgroup
So $\psi$ is a bimodele map. Note that if the Rieffel correspondence pairs up $e_{j}\in \mf{X}$ and $\delta_{x} \in C(X)$ by $e_{j}\cdot \delta_{x}=e_{j}$, then since $1=||\langle e_{j}, e_{j} \rangle_{C(X)}||=||\langle e_{j} \cdot\delta_{x}, e_{j} \cdot\delta_{x} \rangle_{C(X)}||=||\delta_{x} \langle e_{j}, e_{j}\rangle_{C(X)}\delta_{x}||=||\delta_{x} \langle e_{j}, e_{j}\rangle_{C(X)}||$ and $\delta _{x}\delta _{y}=0$ for all $x\neq y$, we must have $ \langle e_{j}, e_{j}\rangle_{C(X)}=\delta_{x}=R(e_{j})$. Likewise, for any $e_{i}\in \mf{X}$ and $\delta_{x}\in C(X)$ paired up by $\delta_{x} \cdot e_{j}=e_{j}$, we must have  $_{C(X)} \langle e_{j}, e_{j}\rangle=\delta_{x}.$ Thus
{%
\abovedisplayskip=0pt \belowdisplayskip=0pt%
\begin{align*}
\langle L(\delta_{x}), L(\delta_{x})\rangle_{C(X)}=R(L(\delta_{x}))=\sigma(\delta_{x})=\delta_{\sigma(x)} \quad \text{and} \quad _{C(X)}\langle L(\delta_{x}), L(\delta_{x})\rangle &=L^{-1}(L(\delta_{x}))\\
&=\delta_{x}
\end{align*}
}%
So if $e, f \in A_{\sigma}$, we have
\begingroup
\allowdisplaybreaks
{%
\abovedisplayskip=0pt \belowdisplayskip=0pt%
\begin{align*}
\langle e, f \rangle_{A} &= \langle \sum\limits_{x\in X}z_{x}\delta_{x}, \sum\limits_{y\in X}z_{y}\delta_{y}\rangle _{A}\\
&= \sum\limits_{x\in X}\overline{z_{x}}z_{x}\delta_{\sigma(x)}\\
&= \sum\limits_{x\in X}\overline{z_{x}}z_{x}\langle  L(\delta_{x}), L(\delta_{x})\rangle _{A}\\
&= \langle \sum\limits_{x\in X}z_{x}L(\delta_{x}), \sum\limits_{y\in X}z_{y}L(\delta_{y})\rangle _{A}\\
&= \langle \psi (\sum\limits_{x\in X}z_{x}\delta_{x}), \psi(\sum\limits_{y\in X}z_{y}\delta_{y}) \rangle _{A}\\
&= \langle \psi(e), \psi(f) \rangle _{A}
\end{align*}
}%
\endgroup
\begingroup
\allowdisplaybreaks
{%
\abovedisplayskip=0pt \belowdisplayskip=0pt%
\begin{align*}
_{A}\langle e, f \rangle &=_{A}\langle \sum\limits_{x\in X}z_{x}\delta_{x}, \sum\limits_{y\in X}z_{y}\delta_{y}\rangle \\
&= \sum\limits_{x\in X}z_{x}\overline{z_{x}}\delta_{x} \\
&= _{A}\langle\sum\limits_{x\in X}z_{x}L(\delta_{x}), \sum\limits_{y\in X}z_{y}L(\delta_{y}) \rangle \\
&= _{A}\langle \psi (\sum\limits_{x\in X}z_{x}\delta_{x}), \psi(\sum\limits_{y\in X}z_{y}\delta_{y}) \rangle _{A}\\
&= _{A}\langle \psi(e), \psi(f) \rangle.
\end{align*}
}%
\endgroup
Thus $\psi$ preserves both inner products (so it is injective). Since $\mathcal{A}$ spans $\mf{X}$, and $\psi(L^{-1}(e_{i}))=L((L^{-1}(e_{i}))=e_{i}$, $\psi$ is surjective. Thus an isomorphism.
\end{proof}
We can also view $_{A}{A_{\sigma}} _{A}$ as a graph correspondence. More precisely, let $G_{\sigma}=(G_{\sigma}^{0}, G_{\sigma}^{1}, r, s)$ be the directed graph given by $G_{\sigma}^{0}=X$, $G_{\sigma}^{1}=\{e_{x}\}_{x\in X}$, $r,s: G_{\sigma}^{1}\to G_{\sigma}^{0}$ given by $r(e_{x})=x$ and $s(e_{x})=\sigma(x)$. Then the graph correspondence  $_{C(G^{0}_{\sigma})}{C(G^{1}_{\sigma}})_{C(G^{0}_{\sigma})}$ associated to $G_{\sigma}$ with the usual actions and inner products:
\begingroup
\allowdisplaybreaks
{%
\abovedisplayskip=0pt \belowdisplayskip=0pt%
\begin{gather*} 
\sum\limits_{x\in X}a_{x}\delta_{e_{x}}\cdot \sum\limits_{y\in X}b_{y}\delta_{y}=\sum\limits_{x\in X}a_{x}b_{\sigma(x)}\delta_{e_{x}}\\
\sum\limits_{y\in X}b_{y}\delta_{y}\cdot \sum\limits_{x\in X}a_{x}\delta_{e_{x}}=\sum\limits_{x\in X}a_{x}b_{x}\delta_{e_{x}}\\
\langle\sum\limits_{x\in X}a_{x}\delta_{e_{x}}, \sum\limits_{y\in X}c_{y}\delta_{e_{y}}\rangle _{C(X)}=\sum\limits_{x\in X}\overline{a_{x}}c_{x}\delta_{\sigma(x)}=\sum\limits_{x\in X}\overline{a_{\sigma ^{-1} (x)}}c_{\sigma ^{-1}(x)}\delta_{x}\\
_{C(X)}\langle\sum\limits_{x\in X}a_{x}\delta_{e_{x}}, \sum\limits_{y\in X}c_{y}\delta_{e_{y}}\rangle=\sum\limits_{x\in X}{a_{x}}\overline{c_{x}}\delta_{x}
\end{gather*}
}%
\endgroup
is isomorphic to $_{A}{A_{\sigma}} _{A}\cong {_{A}\mf{X} _{A}}$ via the map $\omega:C(G_{\sigma}^{1})\to A_{\sigma}$ given by $\omega(\delta_{e_{x}})=\delta_{x}$.
Note that $|G_{\sigma}^{1}|=|G_{\sigma}^{0}|$ and $r,s$ are bijections. So if the graph $G_{\sigma}$ is finite then $G_{\sigma}$ is either a cycle or a disconnected union of cycles (given by the cycle decomposition of $\sigma$). Note also that each permutation $\sigma$ of $X$ gives an equivalence bimodule $_{A}{C(G^{1}_{\sigma}})_{A} \cong {_{A}{A_{\sigma}} _{A}}$.

If $[\mf{X}]$ denotes the isomorphism class of $_{A}\mf{X}_{A}$, then we have:
\begin{lemma} \label{Pgroup}
$P=\{[\mf{X}] : A\stackrel{WME}{\sim}_{\mf{X}}A\}$ is a group with the operation given by $[\mf{X}] *[\mf{Y}] =[\mf{X} \overbar{\otimes}_{A}\mf{Y}] $
\end{lemma}
\begin{proof}
First we show that if $\mf{X}$ and $\mf{Y}$ are $A$-$A$ $W^{*}$-equivalence bimodules, then so is $\mf{X} \overbar{\otimes}_{A}\mf{Y}$. $\mf{X}$ is a $w^{*}$-full right $A$-module, and by Cohen's factorization theorem, $\mf{Y}=A\cdot \mf{Y}$. Thus $\langle \mf{X}, \mf{X} \rangle_{A} \cdot \mf{Y}$ is $w^{*}$-dense in $\mf{Y}$. So $\langle \mf{X} \overbar{\otimes}_{A} \mf{Y}, \mf{X} \overbar{\otimes}_{A} \mf{Y} \rangle_{A}=\langle \mf{Y}, \langle \mf{X}, \mf{X}\rangle_{A}\cdot \mf{Y}\rangle_{A} $ is $w^{*}$-dense in $\langle \mf{Y}, \mf{Y}\rangle_{A}$, which is $w^{*}$-dense in $A$, since $\mf{Y}$ is a $w^{*}$-full right $A$-module. So $\langle \mf{X} \overbar{\otimes}_{A} \mf{Y}, \mf{X} \overbar{\otimes}_{A} \mf{Y} \rangle_{A}$ is $w^{*}$-dense in $A$. Thus $_{A}\mf{X} \overbar{\otimes}_{A}\mf{Y}_{A}$ is a full  right Hilbert $A$-module.  Likewise, $_{A}\mf{X} \overbar{\otimes}_{A}\mf{Y}_{A}$ is a full  left Hilbert $A$-module.

Let $x,y\in _{A}\mf{X} \overbar{\otimes}_{A}\mf{Y}_{A}$ and $a,b\in A$. Then $\langle a\cdot x,y\rangle _{A}=\langle a\cdot (x_{1}\otimes y_{1}),y_{1}\otimes y_{2}\rangle _{A}=\langle a\cdot x_{1}\otimes y_{1},y_{1}\otimes y_{2}\rangle _{A}=\langle y_{1}, \langle a\cdot x_{1},x_{2}\rangle _{A} \cdot y_{2} \rangle _{A}=\langle y_{1}, \langle x_{1},a^{*}\cdot x_{2}\rangle _{A} \cdot y_{2} \rangle _{A}=\langle x_{1}\otimes y_{1}, a^{*}\cdot x_{2} \otimes y_{2}\rangle _{A}=\langle x_{1}\otimes y_{1}, a^{*}\cdot (x_{2} \otimes y_{2}\rangle _{A})=\langle x, a^{*}\cdot y \rangle_{A}$ and $
_{A}\langle x\cdot b,y\rangle= _{A}\langle(x_{1}\otimes y_{1})\cdot b,x_{2}\otimes y_{2}\rangle =_{A}\langle x_{1}\otimes y_{1}\cdot b,x_{2}\otimes y_{2}\rangle=_{A}\langle x_{1},x_{2}\cdot _{A} \langle y_{1} \cdot b, y_{2}\rangle ^{*} \rangle=_{A}\langle x_{1},x_{2}\cdot _{A} \langle y_{1}, y_{2}\cdot b^{*}\rangle ^{*} \rangle=_{A}\langle x_{1}\otimes y_{1},x_{2}\otimes y_{2}\cdot b^{*}\rangle =_{A}\langle x_{1}\otimes y_{1},(x_{2}\otimes y_{2})\cdot b^{*}\rangle=_{A}\langle x,y\cdot b^{*}\rangle$.

Let  $x,y, z\in _{A}\mf{X} \overbar{\otimes}_{A}\mf{Y}_{A}$. Then $_{A}\langle x,y\rangle\cdot z=_{A}\langle x_{1}\otimes y_{1},x_{2}\otimes y_{2}\rangle \cdot x_{3}\otimes y_{3}=_{A}\langle x_{1}\cdot _{A}\langle y_{1}, y_{2}\rangle, x_{2}\rangle \cdot x_{3}\otimes y_{3}=x_{1} \cdot _{A}\langle y_{1}, y_{2}\rangle\cdot \langle x_{2}, x_{3}\rangle _{A} \otimes y_{3}= x_{1}\otimes _{A}\langle y_{1}, y_{2}\rangle \langle x_{2}, x_{3}\rangle _{A} \cdot y_{3}=x_{1}\otimes _{A}\langle y_{1},\langle x_{2}, x_{3}\rangle^{*}_{A}\cdot y_{2}\rangle \cdot y_{3}=x_{1}\otimes y_{1} \cdot \langle \langle x_{2}, x_{3}\rangle^{*}_{A}\cdot y_{2}, y_{3}\rangle=x_{1}\otimes y_{1} \cdot \langle y_{2},  \langle x_{2}, x_{3}\rangle_{A}\cdot y_{3}\rangle _{A}=x_{1}\otimes y_{1} \cdot \langle x_{2}\otimes y_{2},   x_{3}\otimes y_{3}\rangle_{A}=x\cdot \langle y, z  \rangle_{A} $.

Thus $_{A}\mf{X} \overbar{\otimes}_{A}\mf{Y}_{A}$ is a $W^{*}$-equivalence bimodule. Since $_{A}A\overbar{\otimes}_{A} E_{A} \cong _{A} E_{A}\cong _{A} E\overbar{\otimes}_{A} A_{A}$, the identity of $P$ is $[A]$. By \cref{mA}, $[\mf{X}]^{-1}=[\widetilde{\mf{X}}]$. Thus $P$ is a group.
\end{proof}
%This group is usually called the \emph{Picard group} of $A$, and denoted by Pic$(A)$ (See \cite[section 3]{bgr}).
\begin{lemma}\label{justA}
If $\sigma , \tau \in S_{X}$ then $_{A}{A_{\sigma}} _{A} \cong {_{A}{A_{\tau}} _{A}}$ as $W^{*}$-equivalence bimodules.
\end{lemma}
\begin{proof}
If $a=\sum \limits_{x\in X}a_{x}\delta _{x} \in A$, denote $\sum\limits_{x\in X} a_{\sigma (x)}\delta _{x}$ by $a_{\sigma}$. Consider the triple $(\iota, \iota, \pi): _{A}{A_{\sigma}} _{A} \to _{A}{A_{\tau}} _{A}$ where $\iota$ is the identity map on $A$ and $\pi :A\to A$ is given by $\pi(a)=a_{\tau ^{-1} \sigma}$. That is, $ \pi (\sum\limits_{x\in X} z_{x}\delta_{x})=\sum\limits_{x\in X} z_{\tau^{-1}(\sigma(x))}\delta_{x}$. Then
\begingroup
\allowdisplaybreaks
{%
\abovedisplayskip=0pt \belowdisplayskip=1pt%
\begin{align*}
\iota (a\cdot e \cdot b)&=\iota (\sum\limits_{x\in X} a_{x}\delta _{x}\cdot \sum\limits_{x\in X} z_{x}\delta _{x}\cdot \sum\limits_{x\in X} b_{x}\delta _{x})= \iota (\sum\limits_{x\in X} a_{x}z_{x}b_{\sigma (x)}\delta _{x})\\
&=\sum\limits_{x\in X} a_{x}z_{x}b_{\sigma (x)}\delta _{x}= \sum\limits_{x\in X} a_{x}\delta _{x}\cdot \sum\limits_{x\in X} z_{x}\delta _{x}\cdot \sum\limits_{x\in X} b_{\tau ^{-1}(\sigma (x))}\delta _{x}\\
&= \iota (\sum\limits_{x\in X} a_{x}\delta _{x})\cdot \iota(\sum \limits_{x\in X}z_{x}\delta _{x})\cdot \pi(\sum\limits_{x\in X} b_{x}\delta _{x})\\
&=\iota (a)\cdot \iota(e)\cdot \pi(b).
\end{align*}
}%
\endgroup
So $(\iota, \iota, \pi)$ is a bimodule homomorphism.
\begingroup
\allowdisplaybreaks
{%
\abovedisplayskip=0pt \belowdisplayskip=0pt%
\begin{align*}
\langle\iota(\sum\limits_{x\in X} z_{x}\delta _{x}), \iota(\sum\limits_{x\in X} w_{x}\delta _{x})\rangle _{A}&= \langle\sum\limits_{x\in X} z_{x}\delta _{x}, \sum\limits_{x\in X} w_{x}\delta _{x}\rangle _{A}\\
&=\sum\limits_{x\in X}\overline{z_{\tau ^{-1}(x)}} w_{\tau ^{-1}(x)}\delta _{x}\\
&=\pi(\sum\limits_{x\in X}\overline{z_{\sigma ^{-1}(x)}} w_{\sigma ^{-1}(x)}\delta _{x})\\
&= \pi(\langle\sum\limits_{x\in X} z_{x}\delta _{x}, \sum\limits_{x\in X} w_{x}\delta _{x}\rangle _{A})
\end{align*}
}%
\endgroup
and
\begin{align*}
_{A}\langle\iota(\sum\limits_{x\in X} z_{x}\delta _{x}), \iota(\sum\limits_{x\in X} w_{x}\delta _{x})\rangle &= _{A}\langle\sum\limits_{x\in X} z_{x}\delta _{x}, \sum\limits_{x\in X} w_{x}\delta _{x}\rangle=\iota( _{A}\langle\sum\limits_{x\in X} z_{x}\delta _{x}, \sum\limits_{x\in X} w_{x}\delta _{x}\rangle)
\end{align*}
So $(\iota, \iota, \pi)$ preserves inner products. Thus $(\iota, \iota, \pi)$ is a $W^{*}$-equivalence bimodule isomorphism.
\end{proof}
By \cref{Asigma} and \cref{justA}, %the Picard group of $A$ 
$P=\{[\mf{X}] : A\stackrel{WME}{\sim}_{\mf{X}}A\}$ consists of only one element: 
\begin{theorem}
If $A=C(X)$ for some set $X$, then $P=\{[\mf{X}] : A\stackrel{WME}{\sim}_{\mf{X}}A\}=\{[A]\}$.
\end{theorem}
Now consider the $W^{*}$-equivalence bimodule $_{A}{A_{\sigma}}\overbar{\otimes}_{A}{A_{\tau}}_{A}$. Since this bimodule is balanced over $A$, we have that $\delta_{x}\cdot \delta_{z}\otimes \delta_{y}=\delta_{x} \otimes\delta_{z}\cdot \delta_{y}$ if and only if $\sigma(x)=z=y$. Thus the non zero elements of $_{A}{A_{\sigma}}\otimes _{A}{A_{\tau}}_{A}$ are of the form $\sum (z_{x}\delta_{x} \otimes w_{\sigma (x)}\delta_{\sigma (x)})$, where $z_{x}$, $w_{\sigma (x)}\in \mathbb{C} $. Note that if $\sigma ,\tau \in S_{X}$, then \cref{Pgroup} and  \cref{justA} say that $_{A}{A_{\sigma}}\overbar{\otimes}_{A}{A_{\tau}}_{A}$ is isomorphic to $_{A}A _{A}$. Here we give an explicit $W^{*}$-isomorphism between these two $W^{*}$-equivalence bimodules. Consider the triple
\[(\omega, \psi, \pi): {_{A}{A_{\sigma}}\overbar{\otimes}_{A}{A_{\tau}}_{A}} \to {_{A}A _{A}}\]
where $\omega :A\to A$ is given by $\omega (a)=a_{\sigma^{-1}}$ (that is, $\pi (\sum\limits_{x\in X} a_{x}\delta_{x})=\sum\limits_{x\in X} a_{\sigma^{-1}(x)}\delta_{x}$), $\psi :A_{\sigma}\overbar{\otimes}_{A}A_{\tau} \to A$ is given by $\psi (\sum\limits_{x\in X} z_{x}\delta_{x} \otimes_{A} \sum\limits_{x\in X} w_{x}\delta_{x})=\psi (\sum\limits_{x\in X} z_{x}\delta_{x} \otimes w_{\sigma (x)}\delta_{\sigma (x)})=\sum\limits_{x\in X} z_{x}w_{\sigma (x)}\delta_{\sigma (x)}$ and $\pi :A\to A$ is given by $\pi (a)=a_{\tau}$. That is, $\pi (\sum\limits_{x\in X} a_{x}\delta_{x})=\sum\limits_{x\in X} a_{\tau (x)}\delta_{x}$. Then
\begin{align*}
 \psi &(\sum\limits_{x\in X} a_{x}\delta _{x} \cdot (\sum\limits_{x\in X} z_{x}\delta_{x} \otimes_{A} \sum\limits_{x\in X} w_{x}\delta_{x})\cdot \sum\limits_{x\in X} b_{x}\delta _{x} )\\
&=\psi (\sum \limits_{x\in X}a_{x}\delta _{x} \cdot (\sum\limits_{x\in X} z_{x}\delta_{x} \otimes_{A} \sum\limits_{x\in X} w_{\sigma (x)}\delta_{\sigma (x)})\cdot \sum\limits_{x\in X} b_{x}\delta _{x})\\
&=\psi (\sum\limits_{x\in X} a_{x}z_{x}\delta_{x} \otimes w_{\sigma (x)}b_{\tau (\sigma (x))} \delta_{\sigma (x)})\\
&=\sum\limits_{x\in X} a_{x}z_{x}w_{\sigma (x)}b_{\tau (\sigma (x))} \delta_{\sigma(x)}\\
&=\sum \limits_{x\in X}a_{x}\delta _{\sigma(x)} \cdot \sum\limits_{x\in X} z_{x}w_{\sigma (x)}\delta_{\sigma(x)}\cdot \sum\limits_{x\in X} b_{\tau(\sigma(x))}\delta _{\sigma(x)}\\
&=\sum\limits_{x\in X} a_{\sigma ^{-1}(x)}\delta _{x} \cdot \sum\limits_{x\in X} z_{x}w_{\sigma (x)}\delta_{\sigma(x)}\cdot \sum \limits_{x\in X}b_{\tau(x)}\delta _{x}\\
&=\omega (\sum\limits_{x\in X} a_{x}\delta _{x}) \cdot \psi ( \sum\limits_{x\in X} z_{x}\delta _{x} \otimes w_{\sigma (x)}\delta_{\sigma (x)})\cdot \pi (\sum \limits_{x\in X}b_{x}\delta _{x})
\end{align*}
So $(\omega, \psi, \pi)$ is a bimodule homomorphism.
\begingroup
\allowdisplaybreaks
{%
 \belowdisplayskip=1pt%
\begin{align*}
\langle \psi &(\sum\limits_{x\in X} z_{x}\delta_{x} \otimes_{A} \sum\limits_{x\in X} w_{x}\delta_{x}), \psi (\sum\limits_{x\in X} u_{x}\delta_{x} \otimes_{A} \sum\limits_{x\in X} v_{x}\delta_{x})\rangle_{A}\\
&= \langle\psi (\sum\limits_{x\in X} z_{x}\delta_{x} \otimes w_{\sigma(x)}\delta_{\sigma(x)}), \psi (\sum\limits_{x\in X} u_{x}\delta_{x} \otimes v_{\sigma(x)}\delta_{\sigma(x)})\rangle_{A}\\
&= \langle\sum\limits_{x\in X} z_{x}w_{\sigma(x)}\delta_{\sigma(x)}, \sum \limits_{x\in X}u_{x}v_{\sigma(x)}\delta_{\sigma(x)}\rangle_{A}\\
&= \sum\limits_{x\in X} \overline{z_{x}w_{\sigma(x)}}u_{x}v_{\sigma(x)}\delta_{\sigma(x)}\\
&=\pi(\sum \limits_{x\in X}\overline{z_{x}w_{\sigma(x)}}u_{x}v_{\sigma(x)}\delta_{\tau(\sigma(x))})\\
&=\pi(\sum\limits_{x\in X}\langle w_{\sigma(x)}\delta _{\sigma(x)}, \overline{z_{x}}u_{x}v_{\sigma(x)}\delta_{\sigma(x)}\rangle_{A}) \\
&=\pi(\sum\limits_{x\in X}\langle w_{\sigma(x)}\delta _{\sigma(x)}, \overline{z_{x}}u_{x}\delta _{\sigma(x)}\cdot v_{\sigma(x)}\delta_{\sigma(x)}\rangle_{A}) \\
&=\pi(\sum\limits_{x\in X}\langle w_{\sigma(x)}\delta _{\sigma(x)}, \langle z_{x}\delta _{x}, u_{x}\delta _{x}\rangle _{A} \cdot v_{\sigma(x)}\delta_{\sigma(x)}\rangle_{A}) \\
&=\pi(\langle\sum\limits_{x\in X} z_{x}\delta _{x}\otimes w_{\sigma(x)}\delta _{\sigma(x)}, \sum\limits_{x\in X} u_{x}\delta _{x}\otimes v_{\sigma(x)}\delta_{\sigma(x)}\rangle_{A}) \\
&=\pi(\langle\sum\limits_{x\in X} z_{x}\delta _{x}\otimes \sum\limits_{x\in X} w_{x}\delta _{x},\sum\limits_{x\in X} u_{x}\delta _{x}\otimes \sum\limits_{x\in X} v_{x}\delta_{x}\rangle_{A}) \\
\end{align*}
}%
\endgroup
and
\begingroup
\allowdisplaybreaks
{%
\abovedisplayskip=0pt \belowdisplayskip=0pt%
\begin{align*}
_{A}\langle \psi &(\sum\limits_{x\in X} z_{x}\delta_{x} \otimes_{A} \sum\limits_{x\in X} w_{x}\delta_{x}), \psi (\sum\limits_{x\in X} u_{x}\delta_{x} \otimes_{A} \sum\limits_{x\in X} v_{x}\delta_{x})\rangle\\
&=_{A} \langle\psi (\sum\limits_{x\in X} z_{x}\delta_{x} \otimes w_{\sigma(x)}\delta_{\sigma(x)}), \psi (\sum\limits_{x\in X} u_{x}\delta_{x} \otimes v_{\sigma(x)}\delta_{\sigma(x)})\rangle\\
&= _{A}\langle\sum\limits_{x\in X} z_{x}w_{\sigma(x)}\delta_{\sigma(x)}, \sum\limits_{x\in X} u_{x}v_{\sigma(x)}\delta_{\sigma(x)}\rangle\\
&= \sum \limits_{x\in X}z_{x}w_{\sigma(x)}\overline{ u_{x}v_{\sigma(x)}}\delta_{\sigma(x)}\\
&= \omega(\sum\limits_{x\in X} z_{x}w_{\sigma(x)}\overline{ u_{x}v_{\sigma(x)}}\delta_{x})\\
&=\omega(\sum\limits_{x\in X} {_{A}\langle z_{x}\delta _{x}, u_{x}v_{\sigma(x)}\overline{w_{\sigma(x)}}\delta_{x}\rangle}) \\
&=\omega(\sum\limits_{x\in X} {_{A}\langle z_{x}\delta _{x}, u_{x}\delta_{x} \cdot v_{\sigma(x)}\overline{w_{\sigma(x)}}\delta_{\sigma(x)}\rangle}) \\
&=\omega(\sum\limits_{x\in X} {_{A}\langle z_{x}\delta _{x}, u_{x}\delta_{x} \cdot_{A}\langle v_{\sigma(x)}\delta _{\sigma(x)}, w_{\sigma(x)}\delta_{\sigma(x)}\rangle}) \\
&=\omega(\sum\limits_{x\in X} {_{A}\langle z_{x}\delta _{x}\otimes  w_{\sigma(x)}\delta_{\sigma(x)}, u_{x}\delta_{x} \otimes v_{\sigma(x)}\delta _{\sigma(x)}\rangle}) \\
&=\omega( {_{A}\langle\sum\limits_{x\in X} z_{x}\delta _{x}\otimes  w_{\sigma(x)}\delta_{\sigma(x)}, \sum\limits_{x\in X} u_{x}\delta_{x} \otimes v_{\sigma(x)}\delta _{\sigma(x)}\rangle}) \\
&=\omega( {_{A}\langle\sum\limits_{x\in X} z_{x}\delta _{x}\otimes \sum \limits_{x\in X}w_{x}\delta_{x}, \sum\limits_{x\in X} u_{x}\delta_{x} \otimes\sum\limits_{x\in X} v_{x}\delta _{x}\rangle}) \\
\end{align*}
}%
\endgroup
%Note that for any permutation tau, \sum a_{\tau(x)}\delta_{x}=\sum a_{x}\delta_{\tau ^{-1}(x)}
So $(\omega, \psi, \pi): _{A}{A_{\sigma}}\overbar{\otimes} _{A}{A_{\tau}}_{A} \to _{A}A _{A}$ preserves inner products. Thus, it is injective. For each $\sum\limits_{x\in X} a_{x}\delta_{x} \in A$, $\psi(\sum\limits_{x\in X} a_{x}\delta_{\sigma ^{-1}(x)} \otimes \sum\limits_{x\in X} \delta_{x})=\psi(\sum\limits_{x\in X} a_{x}\delta_{\sigma ^{-1}(x)} \otimes \delta_{x})=\sum\limits_{x\in X} a_{x}\delta_{x}$. So $(\omega, \psi, \pi)$ is surjective. Thus a $W^{*}$-equivalence bimodule isomorphism.
\begin{lemma}\label{Moritacoarser}
Let $_{A}E_{A}$ and $_{B}F_{B}$ be two $W^{*}$-correspondences. If $_{A}E_{A}\cong {_{B}F_{B}}$ then  $_{A}E_{A}\stackrel{WME}{\sim} {_{B}F_{B}}$.
\end{lemma}
\begin{proof}
If $_{A}E_{A}\cong _{B}F_{B}$ then there is a $W^{*}$-correspondence isomorphism $(\pi, \phi):_{A}E_{A}\to _{B}F_{B}$, where $\phi$ is a vector space isomorphism and $\pi :A\to B$ is a $W^{*}$-algebra isomorphism. Then $B$ is an $A$-$B$ $W^{*}$-equivalence bimodule with the left action given by $a\cdot b=\pi(a)b$, right action given by multiplication in $B$ and inner products given by $\langle b_{1}, b_{2}\rangle_{B}=b_{1}^{*}b_{2}$ and $_{A}\langle b_{1}, b_{2}\rangle=\pi ^{-1}(b_{1}b_{2}^{*})$.
%If $_{A}E_{A}\cong _{B}F_{B}$ then there is a $W^{*}$-correspondence isomorphism $(\pi, \phi, \pi):_{A}E_{A}\to _{B}F_{B}$, where $\phi$ is a vector space isomorphism and $\pi :A\to B$ is a $W^{*}$-algebra isomorphism. Then $B$ is an $A$-$B$ $W^{*}$-equivalence bimodule with the left action given by $a\cdot b=\pi(a)b$, right action given by multiplication in $B$ and inner products given by $\langle b_{1}, b_{2}\rangle_{B}=b_{1}^{*}b_{2}$ and $_{A}\langle b_{1}, b_{2}\rangle=\pi ^{-1}(b_{1}b_{2}^{*})$.

We show that $(\iota, \varphi):_{B}\widetilde{B}\overbar{\otimes}_{A}E\overbar{\otimes}_{A}B_{B}\to _{B}F_{B}$ is a $W^{*}$-correspondence isomorphism, where $\iota$ is the identity map and $\varphi :\widetilde{B}\overbar{\otimes}_{A}E\overbar{\otimes}_{A}B\to F$ is defined by $\varphi(\widetilde{b}\otimes e \otimes c)=b^{*}\cdot\phi(e)\cdot c$. Let $e,g\in E$ and $ a,b,c,d, \alpha, \beta \in B$. Then
\begin{align*}
\varphi(\alpha\cdot(\widetilde{b}\otimes e \otimes c)\cdot\beta)&=\varphi(\widetilde{b\cdot\alpha^{*}}\otimes e\otimes c\cdot \beta)=\varphi(\widetilde{b\alpha^{*}}\otimes e\otimes c\beta)\\
&=\alpha b^{*}\cdot\phi(e)\cdot c\beta=\alpha \cdot ( b^{*}\cdot\phi(e)\cdot c)\cdot \beta\\
&=\alpha \cdot \varphi( \widetilde{b}\otimes e\otimes  c)\cdot \beta
\end{align*}
So $(\iota,\varphi,\iota)$ is a correspondence homomorphism.
\begin{equation*}
\begin{split}
\langle\varphi &(\widetilde{b}\otimes e\otimes c), \varphi(\widetilde{a}\otimes g\otimes d)\rangle_{B}\\
&=\langle b^{*}\cdot \phi(e)\cdot c, a^{*}\cdot \phi(g)\cdot d\rangle_{B}
=c^{*}\langle b^{*}\cdot \phi(e), a^{*}\cdot \phi(g)\rangle_{B}d\\
&=c^{*}\langle  \phi(e),b\cdot a^{*}\cdot \phi(g)\rangle_{B}d
=c^{*}\langle  \phi(e),ba^{*}\cdot \phi(g)\rangle_{B}d\\
&=c^{*}\langle  \phi(e), \pi(\pi^{-1}(ba^{*}))\cdot \phi(g)\rangle_{B}d=c^{*}\langle  \phi(e), \phi(\pi^{-1}(ba^{*})\cdot g)\rangle_{B}d\\
&=c^{*}\pi(\langle e, \pi^{-1}(ba^{*})\cdot g\rangle_{A})d
%since(\pi,\phi) is an isomorphism, it preserves the inner product
=\langle c, \pi(\langle e, \pi^{-1}(ba^{*})\cdot g\rangle_{A})d\rangle_{B}\\
&=\langle c, \langle e, \pi^{-1}(ba^{*})\cdot g\rangle_{A}\cdot d\rangle_{B}
%since this is how the left action of A on B is defined
=\langle c, \langle e, _{A}\langle b,a\rangle\cdot g\rangle_{A}\cdot d\rangle_{B}\\
&=\langle c, \langle e, \langle\widetilde{b},\widetilde{a}\rangle_{A}\cdot g\rangle_{A}\cdot d\rangle_{B}=\langle c, \langle \widetilde{b}\otimes e, \widetilde{a}\otimes g\rangle_{A}\cdot d\rangle_{B}\\
&=\langle \widetilde{b}\otimes e\otimes c, \widetilde{a}\otimes g\otimes d\rangle_{B}=\iota(\langle \widetilde{b}\otimes e\otimes c, \widetilde{a}\otimes g\otimes d\rangle_{B})
\end{split}
\end{equation*}
So $(\iota,\varphi,\iota)$ preserves the inner product. Thus it is injective. Since for each $f\in F$, there is $e\in E$ such that $\phi(e)=f$, we have $\varphi(1\otimes e\otimes 1)=f$. So $\varphi$ is surjectve, thus a $W^{*}$-correspondence isomorphism.

Since $_{B}\widetilde{B}\overbar{\otimes}_{A}E\overbar{\otimes}_{A}B_{B}\cong {_{B}F_{B}}$, we have $_{A}B\overbar{\otimes}_{B}F \cong _{A}B\overbar{\otimes}_{B}\widetilde{B}\overbar{\otimes}_{A}E\overbar{\otimes}_{A}B_{B}\cong {_{A}A}\overbar{\otimes}_{A}E\overbar{\otimes}_{A}B_{B}\cong {_{A}E\overbar{\otimes}_{A}B_{B}}$. Thus  $_{A}E_{A}\stackrel{WME}{\sim}{_{B}F_{B}}$.
\end{proof}
%One natural question we can ask about $W^{*}$-graph correspondences is the following: If we have two random directed graphs $G$ and $F$, is there a way of identifying if the two $W^{*}$-correspondences derived from $G$ and $F$, will have weakly Morita equivalent Hardy algebras? The answer to this question will come as a corollary of some of the following results.
\begin{theorem} \label{moritaiso}
If $_{A}E_{A}$ and $_{B}D_{B}$ are $W^{*}$-graph correspondences then $_{A}E_{A}\stackrel{WME}{\sim}{_{B}D_{B}}$ if and only if $_{A}E_{A}\cong {_{B}D_{B}}$.
\end{theorem}
\begin{proof}
One direction was already shown in \cref{Moritacoarser}. Now we show the converse. If $_{A}E_{A}\stackrel{WME}{\sim}{_{B}D_{B}}$ then we have $A\stackrel{WME}{\sim}B$, and since $A$ and $B$ are commutative, we have $A\cong B$ (recall that if two $W^{*}$-algebras are Morita equivalent then their centers are isomorphic). % by [Proposition 8.1]{Riefell74}.
So there is a $W^{*}$-algebra isomorphism $\alpha: B\to A$, such that $(\alpha, \iota):{_{B}D_{B}} \to {_{A}D_{A}}$ is a $W^{*}$-correspondence isomorphism. Then by \cref{Moritacoarser}, we have $_{A}E_{A}\stackrel{WME}{\sim}{_{A}D_{A}}$. By \cref{Asigma}, a $W^{*}$-equivalence bimodule $_{A}\mathfrak{X}_{A}$ is isomorphic to ${_{A}{A_{\sigma}} _{A}}$, where $ _{A}{A_{\sigma}} _{A}$ is the same as $ _{A}A _{A}$ but with a modified right action and right inner product determined by some permutation $\sigma$ of $S$. Then $_{A}E_{A}\stackrel{WME}{\sim}_{A}D_{A}$ implies $A_{\sigma}\overbar{\otimes}_{A}D \cong E\overbar{\otimes}_{A}A_{\sigma}$. So using \cref{mA}, we have
\begin{equation*}
\begin{split}
_{B}D_{B}\cong  _{A}D_{A}\cong A\overbar{\otimes}_{A}D\cong  \widetilde{A_{\sigma}}\overbar{\otimes}_{A} A_{\sigma}\overbar{\otimes}_{A}D \cong \widetilde{A_{\sigma}}\overbar{\otimes}_{A}E\overbar{\otimes}_{A}A_{\sigma}
\end{split}
\end{equation*}
Thus, to show that $_{A}E_{A}\cong {_{A}D_{A}}$, all we need to show is that $\widetilde{A_{\sigma}}\overbar{\otimes}_{A}E\overbar{\otimes}_{A}A_{\sigma}\cong {_{A}E_{A}}$.

Consider the pair $(\pi, \phi)$ where $\phi :\widetilde{A_{\sigma}}\overbar{\otimes}_{A}E\overbar{\otimes}_{A}A_{\sigma}\to _{A}E_{A}$ is given by $\phi (\widetilde{a} \otimes x \otimes b)=a^{*} \cdot x \cdot b$ and $\pi :A\to A$ is given by $\pi (c)=c_{\sigma}$. That is, $\pi (\sum \limits_{x\in X} c_{i}\delta_{v_{i}})=\sum \limits_{x\in X} c_{\sigma(i)}\delta_{v_{i}}$. Clearly, $\pi$ is a $W^{*}$-isomorphism. Now we show that $(\pi, \phi):_{A}\widetilde{A_{\sigma}}\overbar{\otimes}_{A}E\overbar{\otimes}_{A}{A_{\sigma}}_{A}\to _{A}E_{A}$ is a $W^{*}$-correspondence isomorphism. Let $a,b,c,d \in A_{\sigma}$, $\alpha, \beta \in A$ and $x,y \in E$.
\begin{align*}
\phi (\alpha\cdot(\widetilde{a} \otimes x \otimes b)\cdot \beta)&=\phi ((\alpha\cdot\widetilde{a}) \otimes x \otimes (b\cdot \beta))
=\phi ((\widetilde{a\cdot \alpha^{*}}) \otimes x \otimes (b\cdot \beta))\\
&=\phi ((\widetilde{a \alpha^{*}_{\sigma}}) \otimes x \otimes (b\beta_{\sigma}))
=\alpha_{\sigma}a^{*}\cdot x\cdot b\beta_{\sigma}\\
&=\alpha_{\sigma}\cdot (a^{*}\cdot x\cdot b)\cdot \beta_{\sigma}=\pi (\alpha)\cdot \phi (\widetilde{a} \otimes x \otimes b)\cdot \pi(\beta)
\end{align*}
and
\begin{align*}
\langle\phi (\widetilde{a} \otimes x \otimes b), \phi (\widetilde{c} \otimes y \otimes d)\rangle _{A}&=
\langle a^{*} \cdot x \cdot b, c^{*} \cdot y \cdot d\rangle _{A}= b^{*}\langle a^{*} \cdot x, c^{*} \cdot y\rangle  _{A}d\\
&= \pi((b^{*}\langle a^{*}\cdot x, c^{*} \cdot y \rangle _{A}d)_{\sigma ^{-1}})= \pi((b^{*}\langle x, ac^{*} \cdot y \rangle _{A}d)_{\sigma ^{-1}})\\
&= \pi((b^{*}\langle ac^{*} \cdot y, x \rangle^{*} _{A}d)_{\sigma ^{-1}})= \pi(((\langle ac^{*} \cdot y, x \rangle _{A}b)^{*}d)_{\sigma ^{-1}})\\
&= \pi(\langle\langle ac^{*} \cdot y, x \rangle _{A}b, d \rangle_{A}) %the elements inside the outer inner product are elements in A_sigma and the inner product in A_sigma is defined by taking the product of the sigma^{-1} coefficients (using the adjoint of the first element in the inner product)
=\pi(\langle\langle ac^{*} \cdot y, x \rangle _{A}\cdot b, d \rangle_{A})\\
&=\pi(\langle\langle x, ac^{*} \cdot y \rangle^{*} _{A}\cdot b, d \rangle_{A})=\pi(\langle b, \langle x, ac^{*} \cdot y\rangle_{A}\cdot d \rangle_{A})\\
&=\pi(\langle b, \langle x, _{A}\langle a, c \rangle \cdot y\rangle_{A}\cdot d \rangle_{A})=\pi(\langle b, \langle x, \langle \widetilde{a}, \widetilde{c} \rangle _{A} \cdot y\rangle_{A}\cdot d \rangle_{A})\\
&=\pi(\langle b, \langle \widetilde{a} \otimes x, \widetilde{c} \otimes y\rangle _{A} \cdot d\rangle  _{A})=
\pi(\langle \widetilde{a} \otimes x \otimes b, \widetilde{c} \otimes y \otimes d\rangle _{A})
\end{align*}
So $\phi$ is isometric, thus injective. Since for each $e\in E$, $\phi(1_{A}\otimes e\otimes 1_{A})=1\cdot e\cdot 1=e$, $\phi$ is surjective. Thus $(\pi, \phi)$ is a $W^{*}$-correspondence isomorphism.
\end{proof}
%the elements inside the outer inner product are elements in A_sigma and the inner product in A_sigma is defined by taking the product of the sigma^{-1} coefficients (using the adjoint of the first element in the inner product)
Two directed graphs $G=(G^{0}, G^{1}, s_{1}, r_{1})$ and $F=(F^{0}, F^{1}, s_{2}, r_{2})$ are \emph{isomorphic} if there are two bijections $\alpha :G^{1} \to F^{1}$ and $\beta :G^{0}\to F^{0}$ such that for each edge $e\in G^{1}$, $s_{2}(\alpha(e))=\beta(s_{1}(e))$ and $r_{2}(\alpha(e))=\beta(r_{1}(e))$. 

Clearly, if we draw a directed graph $G=(G^{0}, G^{1}, s_{1}, r_{1})$  and relabel its edges and its vertices then we produce a new graph $F=(F^{0}, F^{1}, s_{2}, r_{2})$ whose identical drawing implies that the two  relabeling bijections $\gamma :G^{1} \to F^{1}$ and $\lambda: G^{0}\to F^{0}$ satisfy $s_{2}(\gamma(e))=\lambda(s_{1}(e))$ and $r_{2}(\gamma(e))=\lambda(r_{1}(e))$. So we obtain an isomorphic graph. In particular, if $G^{1}=F^{1}$ and $G^{0}=F^{0}$ then $\gamma$ and $\lambda$ are permutations.
\begin{theorem} \label{moritagraph}
Let $_{A}E_{A}$ and $_{B}D_{B}$ be $W^{*}$-graph correspondences associated to the directed graphs $G=(G^{0}, G^{1}, s_{1}, r_{1})$ and $F=(F^{0}, F^{1}, s_{2}, r_{2})$ respectively. $_{A}E_{A} \cong {_{B}D_{B}}$ if and only if $G\cong F$.
% there are three bijections $\alpha :G^{1} \to F^{1}$ and $\beta, \gamma :G^{0}\to F^{0}$ such that for each edge $e_{i}\in G^{1}$, $s_{2}(\alpha(e_{i}))=\gamma(s_{1}(e_{i}))$ and $r_{2}(\alpha(e_{i}))=\beta(r_{1}(e_{i})$. 
\end{theorem}%If we allow the definition of correspondence isomorphism to use triples, as Muhly and Solel allow on the Morita tensor algebras paper pg 115, then this theorem is right. The problem with allowing the triple definition is that when we take tensor products then since they are balanced , problems arise.
\begin{proof}
First note that $G\cong F$ is a particular case of having three bijections $\alpha :G^{1} \to F^{1}$ and $\beta, \gamma :G^{0}\to F^{0}$ such that for each edge $e_{i}\in G^{1}$, $s_{2}(\alpha(e_{i}))=\gamma(s_{1}(e_{i}))$ and $r_{2}(\alpha(e_{i}))=\beta(r_{1}(e_{i})$. More precisely, $G\cong F$ is the special case when $\beta=\gamma$.

If $G$ and $F$ are isomorphic graphs, then there are two bijections $\alpha :G^{1} \to F^{1}$ and $\beta :G^{0}\to F^{0}$ such that for each edge $e_{i}\in G^{1}$, $s_{2}(\alpha(e_{i}))=\beta(s_{1}(e_{i}))$ and $r_{2}(\alpha(e_{i}))=\beta(r_{1}(e_{i})$ then let $\varphi:E\to D$ be given by $\varphi (\delta_{e_{i}})=\delta_{\alpha(e_{i})}$ and $\omega:A\to B$ be given by $\omega(\delta_{v_{i}})=\delta_{\beta(v_{i})}$. Then
\begin{align*}
\varphi(\delta_{r_{1}(e_{i})}\cdot\delta_{e_{i}}\cdot\delta_{s_{1}(e_{i})})=\varphi(\delta_{e_{i}})=\delta_{\alpha (e_{i})}
&=\delta_{r_{2}(\alpha(e_{i}))}\cdot\delta_{\alpha (e_{i})}\cdot\delta_{s_{2}(\alpha(e_{i}))}\\
&=\delta_{\beta (r_{1}(e_{i}))}\cdot\delta_{\alpha (e_{i})}\cdot\delta_{\beta (s_{1}(e_{i}))}\\
&=\omega(\delta_{r_{1}(e_{i})})\cdot\varphi(\delta_{e_{i}})\cdot\omega(\delta_{s_{1}(e_{i})})
\end{align*}
and
\begin{align*}
\langle\varphi(\delta_{e_{i}}),\varphi(\delta_{e_{i}})\rangle_{A}&=\langle\delta_{\alpha(e_{i})},\delta_{\alpha(e_{i})}\rangle_{A}=\delta_{s_{2}(\alpha(e_{i}))}=\delta_{\beta(s_{1}((e_{i}))}=\omega(\delta_{s_{1}(e_{i})})=\omega(\langle\delta_{e_{i}},\delta_{e_{i}}\rangle_{A})
\end{align*}
So
\begin{align*}
\varphi(a\cdot x\cdot b) &= \varphi (\sum\limits_{k}a_{k}\delta _{v_{k}}\cdot \sum\limits_{i}z_{i}\delta _{e_{i}}\cdot \sum\limits_{j}b_{j}\delta _{v_{j}})\\
&=\varphi (\sum\limits_{k}\sum\limits_{i}\sum\limits_{j}a_{k}z_{i}b_{j}\delta _{v_{k}}\cdot \delta _{e_{i}}\cdot \delta _{v_{j}})\\
&= \varphi (\sum_{\substack{r_{1}(e_{i})=v_{k}\\ 
s_{1}(e_{i})=v_{j}}}a_{k}z_{i}b_{j}\delta _{e_{i}})\\
&= \sum_{\substack{r_{2}(\alpha( e_{i}))=\beta( v_{k})\\ 
s_{2}(\alpha( e_{i}))=\beta( v_{j})}}a_{k}z_{i}b_{j}\delta _{\alpha (e_{i})}\\
&=\sum\limits_{k}\sum\limits_{i}\sum\limits_{j}a_{k}z_{i}b_{j}\delta _{\beta( v_{k})}\cdot \delta _{\alpha( e_{i})}\cdot \delta _{\beta ( v_{j})}\\
&= (\sum\limits_{k}a_{k}\delta _{\beta( v_{k})})\cdot (\sum\limits_{i}z_{i}\delta _{\alpha( e_{i})})\cdot (\sum\limits_{j}b_{j}\delta _{\beta( v_{j})})\\
&=\omega( a)\cdot \varphi( x)\cdot \omega( b)
\end{align*}
and
\begin{align*}
\langle\varphi(x), \varphi(y)\rangle _{B} &= \langle\sum\limits_{i}z_{i}\delta _{\alpha( e_{i})}, \sum\limits_{j}y_{j}\delta _{\alpha( e_{j})}\rangle_{B}= \sum \limits_{i,j}\overline{z_{i}}\langle\delta _{\alpha( e_{i})}, \delta _{\alpha( e_{j})}\rangle_{B} y_{j} \\
&= \sum \limits_{i=j}\overline{z_{i}}\delta _{s_{2}(\alpha( e_{i}))} y_{j} = \sum \limits_{i=j}\overline{z_{i}}\delta _{\beta(s_{1}( e_{i}))} y_{j} = \sum \limits_{i=j}\omega(\overline{z_{i}}\delta _{s_{1}( e_{i})} y_{j}) \\
&= \omega( \sum \limits_{i=j}(\overline{z_{i}}\delta _{s_{1}( e_{i})} y_{j})= \omega( \sum \limits_{i, j}(\overline{z_{i}}\langle\delta _{ e_{i}}, \delta _{ e_{j}}\rangle_{A} y_{j}) = \omega( \sum \limits_{i, j}\langle z_{i}\delta _{ e_{i}}, y_{j} \delta _{ e_{j}}\rangle_{A} ) \\
&= \omega( \sum \limits_{i, j}\langle z_{i}\delta _{ e_{i}}, y_{j} \delta _{ e_{j}}\rangle_{A} ) = \omega( \langle \sum \limits_{i} z_{i}\delta _{ e_{i}}, \sum \limits_{j} y_{j} \delta _{ e_{j}}\rangle_{A} ) \\
&= \omega(\langle x, y\rangle _{A} )
\end{align*}
Since for each $\delta_{e_{i}}\in D$, $\varphi(\delta_{\alpha^{-1}(e_{i})})=\delta_{e_{i}}$, $\varphi$ is surjective. Thus $(\omega, \varphi):_{A}E_{A}\to _{B}D_{B}$ is a $W^{*}$-correspondence isomorphism.

Now we show the converse. If $_{A}E_{A} \cong {_{B}D_{B}}$, then there is a $W^{*}$-correspondence isomorphism  $(\omega, \varphi):_{A}E_{A}\to _{B}D_{B}$. Since $\omega$ and $\varphi$ are bijections, we have $|G^{1}|=|F^{1}|$ and $|G^{0}|=|F^{0}|$. Since relabeling vertices and edges gives an isomorphic graph, we may assume that $G^{0}=F^{0}$ and $G^{1}=F^{1}$. Since each $\delta _{v_{i}}$ is a projection, $\omega(\delta _{v_{i}})=\omega(\delta _{v_{i}}^{n})=\omega(\delta _{v_{i}})^{n}$ for all positive integers $n$. So $\omega(\delta _{v_{i}})$ is of the form $\sum\limits_{j}\delta _{v_{j}}$, and since $\omega$ is an isometry, we have $\omega(\delta _{v_{i}})=\delta _{v_{t}}$ for some vertex $v_{t}$.

Let $e_{i}\in G^{1}$,  $\omega(\delta_{s_{1}(e_{i})})= \delta_{v_{k}}$ and $\varphi(\delta_{e_{i}})=\sum z_{j}\delta _{e_{j}}$. Since $|\varphi(\delta_{e_{i}})|=1$ (being an isometry) and $\varphi(\delta_{e_{i}})=\varphi(\delta_{e_{i}}\cdot\delta_{s(e_{i})})=\varphi(\delta_{e_{i}})\cdot\omega(\delta_{s(e_{i})})=(\sum z_{j}\delta _{e_{j}})\cdot(\delta_{v_{k}})=z\delta_{s_{2}^{-1}(v_{k})}$, % $||\varphi(\delta_{e_{i}})||^{2}=||\langle \varphi(\delta_{e_{i}}), \varphi(\delta_{e_{i}})\rangle||=||\pi(\langle\delta_{e_{i}}, \delta_{e_{i}}\rangle) =||\pi(\delta_{s(e_{i})})||=1$
we must have, $\varphi(\delta_{e_{i}})=\delta_{s_{2}^{-1}(v_{k})}$. Thus $\omega$ and $ \varphi$ are given by permutations.
 %up to change of signs. Note that since $\varphi(\delta_{e_{i}})=\varphi(\delta_{r(e_{i})}\cdot \delta_{e_{i}}\cdot\delta_{s(e_{i})})=\omega(\delta_{r(e_{i})})\cdot\varphi(\delta_{e_{i}})\cdot\pi(\delta_{s(e_{i})})$, we have that $\omega(\delta_{r(e_{i})})$ and $\pi(\delta_{s(e_{i})})$ always have the same sign. Since the $\pm$ signs of each $\varphi(\delta_{e_{i}})$ and $\omega(\delta_{v_{j}})$ and $\pi(\delta_{v_{j}})$ don't make any difference in the graph of $F$, we may assume they are always positive.
Let $\beta:G^{0}\to F^{0}$ be the permutation given by $\beta(v_{j})=v_{k}$ if $\omega(\delta_{v_{j}})=\delta_{v_{k}}$. Let $\alpha:G^{1}\to F^{1}$ be the permutation given by $\alpha(e_{i})=e_{m}$ if $\varphi(\delta_{e_{i}})=\delta_{e_{m}}$. In ${_{B}D_{B}}$, we have:

\noindent $\delta_{\beta(r_{1}(e_{i}))}\cdot\delta_{\alpha(e_{i})}\cdot\delta_{\beta(s_{1}(e_{i}))}=\omega(\delta_{r_{1}(e_{i})})\cdot\varphi(\delta_{e_{i}})\cdot\omega(\delta_{s_{1}(e_{i})})=\varphi(\delta_{r_{1}(e_{i})} \cdot\delta_{e_{i}}\cdot\delta_{s_{1}(e_{i})})=\varphi(\delta_{e_{i}})=\delta_{\alpha(e_{i})}$. So $\beta(s_{1}(e_{i}))=s_{2}(\alpha(e_{i}))$ and $\beta(r_{1}(e_{i}))=r_{2}(\alpha(e_{i}))$. $G$ anf $F$ are isomorphic graphs.
\end{proof}
\begin{corollary}\label{graphpermutationsmorita}
Let $_{A}E_{A}$ and $_{B}D_{B}$ be $W^{*}$-graph correspondences associated to the directed graphs $G=(G^{0}, G^{1}, s_{1}, r_{1})$ and $F=(F^{0}, F^{1}, s_{2}, r_{2})$ respectively. $_{A}E_{A} \stackrel{WME}{\sim}$ $_{B}D_{B}$ if and only if $G\cong F$.
% there are three bijections $\alpha :G^{1} \to F^{1}$ and $\beta, \gamma :G^{0}\to F^{0}$ such that for each edge $e_{i}\in G^{1}$, $s_{2}(\alpha(e_{i}))=\gamma(s_{1}(e_{i}))$ and $r_{2}(\alpha(e_{i}))=\beta(r_{1}(e_{i})$.
\end{corollary}
\begin{proof}
It follows from \cref{moritaiso} and \cref{moritagraph}.
\end{proof}

\bibliographystyle{amsalpha}
\bibliography{ReneMaster}

\end{document}